%% file: Survey_ITCS.tex
\documentclass[leqno,11pt]{amsart}
\usepackage{amsaddr}

\usepackage[utf8x]{inputenc}
\usepackage{microtype}
\usepackage[margin=2cm]{geometry}
\usepackage[all]{xy}
\usepackage{amsmath} 
\usepackage{amsthm,amssymb,mathrsfs} 
\usepackage[T1]{fontenc}
\usepackage{ae,aecompl}
\usepackage{times}
\usepackage[english]{babel}
\usepackage[colorlinks,pagebackref,hypertexnames=false]{hyperref}
\usepackage[alphabetic,backrefs]{amsrefs}
\usepackage{tikz}
\usetikzlibrary{matrix,decorations.pathreplacing,calc}

\numberwithin{equation}{section}

\input{init}

\newcommand{\squig}{\rightsquigarrow}
\newcommand{\squigd}{\overset{\delta}{\rightsquigarrow}}

\newcommand{\kd}{\Sigma}
\newcommand{\ignore}[1]{}
\newcommand{\gtmfd}{$\rG_{2}$\nobreakdash-\hspace{0pt}manifold}
\newcommand{\ie}{\emph{i.e.} }

\newcommand{\cf}{\emph{cf.} }
\newcommand{\wt}{\widetilde}
\newcommand{\nres}{R}
\newcommand{\hk}{hyper-K\"ahler\ }
\DeclareMathAlphabet{\df}{U}{eus}{m}{n}
\DeclareMathAlphabet{\matheur}{U}{eur}{m}{n}
\newcommand{\hkr}{\fr}
\newcommand{\kclass}{\matheur{k}}
\newcommand{\gtstr}{$\rG_{2}$\nobreakdash-\hspace{0pt}structure}
\DeclareMathOperator{\Endr}{\mathscr{E}\textit{nd}}

\DeclareMathOperator{\res}{res}
\newcommand{\acylcy}{$\mathrm{ACylCY}^3$}
\newcommand{\acyl}{$\mathrm{ACyl}$}

\author{Henrique N. S\'a Earp}
\address{University of Campinas (Unicamp)}
\thanks{The author is supported by São Paulo Research Foundation (FAPESP) grants 2017/06298-2 and 2017/20007-0 and Brazilian National Research Council (CNPq) grant 312390/2014-9}
\title{Current progress on $\rG_2$--instantons over twisted connected sums}

\date{\today}

\begin{document}

\maketitle\input{abs}
\tableofcontents

\section{Introduction}

This text addresses the existence problem of $\rm G_2$--instantons over twisted connected
sums, as formulated by Walpuski and myself in  \cite{SaEarp2015b}, and the production of
the first examples to date of solutions obtained by a nontrivially \emph{transversal}
gluing process \cite{Menet2017}.  It is aimed at graduate students and researchers in nearby areas who might be interested in a condensed exposition of the main results  spread over my articles \cite{SaEarp2015a,SaEarp2015b,Jardim2017,Menet2017} with Walpuski, Menet at al. and Menet-Nordstr\"om. By no means should this survey convey the impression that the subject is somehow closed or even in its best notational setup; indeed there is much ongoing work on this topic. A number of important questions remain open and the most impressive expected results in this theory are surely still ahead of us.

Recall that a \emph{$\rm G_2$--manifold} $(X,g_{\phi})$ is a Riemannian $7$--manifold together with a torsion-free \emph{$\rm G_2$--structure}, that is, a non-degenerate closed $3$--form $\phi$ satisfying a certain non-linear partial differential equation;
in particular, $\phi$ induces a Riemannian metric $g_\phi$ with
$\Hol(g_{\phi})\subset \rm G_2$ \cite{Joyce1996}*{Part I}.
A \emph{$\rm G_2$--instanton} is a connection $A$ on some $G$--bundle $E\to X$ such that $F_A\wedge*\phi=0$.
Such solutions have a well-understood elliptic deformation theory of index
$0$ \cite{SaEarp2009}, and some form of `instanton count' of their
moduli space is expected to yield new invariants of $7$--manifolds, much
in the same vein as the Casson invariant and instanton Floer homology from flat connections on $3$--manifolds \cite{Donaldson1990, Donaldson1998}. While some important analytical groundwork has been established towards that goal \cite{Tian2000},
major compactification issues remain and this suggests that a thorough understanding of
the general theory might currently have to be postponed in favour of  exploring a good number of functioning examples. This  article proposes a method to construct  such examples.

Readers interested in a more detailed account of instanton theory on $\rm G_2$--manifolds are kindly
 referred to the introductory sections of \cite{SaEarp2015a,SaEarp2015b} and works cited therein.

An important method to produce examples of compact  $\rG_2$--manifolds with $\Hol(g)=\rG_2$ is the \emph{twisted connected sum construction}, suggested by Donaldson, pioneered by Kovalev~\cite{Kovalev2003} and later extended and improved by Kovalev--Lee~\cite{Kovalev2011} and Corti--Haskins--Pacini-Nordström~\cite{Corti2015}.
Here is a brief summary of this construction:
A \emph{building block} consists of a projective $3$--fold $Z$ and a smooth anti-canonical $K3$ surface $\Sigma\subset Z$ with trivial normal bundle (cf. Definition~\ref{def:building-block}).
Given a choice of hyperkähler structure $\(\omega_I,\omega_J,\omega_K\)$ on $\Sigma$ such that $[\omega_I]$ is the restriction of a Kähler class on $Z$, one can make $V:=Z\setminus \Sigma$ into an asymptotically cylindrical (\acyl) Calabi--Yau $3$--fold, that is, a non-compact Calabi--Yau $3$--fold with a tubular end modelled on $\R_+\times  \mathbb{S}^1\times \Sigma$, see Haskins--Hein--Nordström~\cite{Haskins2015}.
Then $Y:= \mathbb{S}^1\times V$ is an \acyl $\rG_2$--manifold with a tubular end modelled on  $\R_+\times \mathbb{T}^2\times \Sigma$.


When a pair $(Z_\pm,\Sigma_\pm)$ of building blocks \emph{matches} `at infinity', in a suitable sense, one can glue $Y_\pm$ by interchanging the $ \mathbb{S}^1$--factors.
This yields a simply-connected compact $7$--manifold $Y$ together with a family of torsion-free $\rG_2$--structures $(\phi_T)_{T \geq T_0}$, see Kovalev \cite{Kovalev2003}*{\S~4}.
From the Riemannian viewpoint $(Y,\phi_T)$ contains a ``long neck'' modelled on $[-T,T]\times \mathbb{T}^2\times \Sigma_+$; one can think of the twisted connected sum as reversing the degeneration of the family of $\rG_2$--manifolds that occurs as the neck becomes infinitely long.
In \cite{Corti2013,Corti2015,Kovalev2003}, building blocks $Z$ are produced
by blowing up
Fano or semi-Fano 3-folds along the base curve $\sC$ of an anticanonical
pencil
(cf. Proposition \ref{FanoBlock}). By
understanding the deformation theory of pairs $(X,\kd)$ of semi-Fanos $X$
and anticanonical K3 divisors $\kd \subset X$, one can produce hundreds of
thousands of pairs with the required matching (see \S \ref{subsec:match}).

This construction raises a natural programme in gauge theory, aimed at
constructing $\rm G_2$-instantons over compact manifolds obtained as a TCS,
originally
outlined in \cite{SaEarp2009}. If $(Z,\Sigma)$ is a building block and $\cE\to
Z$ holomorphic bundle such that $\cE|_\Sigma$ is stable, then $\cE|_\Sigma$
carries a unique ASD instanton compatible with the holomorphic structure~\cite{Donaldson1985}.
In this situation $\cE|_V$ can be given a Hermitian--Yang--Mills (HYM) connection
asymptotic to the ASD instanton on $\cE|_\Sigma$    \cite[Theorem 58]{SaEarp2015a},
whose pullback over $V$ to $ \mathbb{S}^1\times V$ is a \emph{$\rG_2$--instanton},
i.e., a connection $A$ on a $G$--bundle over a $\rG_2$--manifold such that
$F_A\wedge\psi=0$ with $\psi:=*\phi$. It is possible to glue 
a hypothetical pair of such solutions into a $\rm G_2$-instanton over the
\emph{compact} twisted connected sum, provided a number of technical conditions
are met (cf. Theorem \ref{thm:itcs}).

However, the hypotheses of  our  $\rm G_2$-instanton gluing theorem
are rather restrictive and it is not immediate to obtain suitable holomorphic bundles   $\cE_\pm\to Z_\pm$ over the
matching blocks. In particular, a transversality condition over the $K3$ surface $\Sigma_\pm$ `at infinity' requires
some more thorough understanding of the deformation theory of data $(Z_\pm, \kd_\pm, \cE_\pm)$.
Assuming the so-called \emph{rigid} case in which the instantons
that
are glued are isolated points in their moduli spaces, Walpuski \cite{Walpuski2016} was able to exhibit one such example. The trade-off comes in the forms of further constraints to the matching problem for the building
blocks, which makes that \emph{ad hoc} approach difficult to generalise.

Finally, in \cite{Menet2017}, we use the Hartshorne-Serre construction (cf. Theorem \ref{thm: Hartshorne-Serre}) to obtain families of bundles over
the
building blocks. Our method allows one  to generate a  large number
of examples for which the gluing is \emph{nontrivially transversal} (see \S \ref{sec: general algorithm}).
These are particularly relevant, because they open the possibility of
obtaining a conjectural instanton number on the $\rm G_2$-manifold $X$ as
a genuine Lagrangian intersection
within the moduli space $\sM_{S_+}$ over the $K3$ cross-section along the neck,
which can be addressed
by enumerative methods in the future. 

\newpage
\section{Background on $\rG_2$-geometry}

Let us recall some $\rG_2$-trivia, following the exposition in \cite{SaEarp2014}; of course the immortal introductory references for the topic are  \cite{Bryant1985,Salamon1989,Joyce2000}.
Recall that a $\rG_2-$\emph{structure} on an oriented smooth $7-$manifold $Y$  is a smooth $3-$form $\phi \in \Omega ^{3}\left( Y\right) $ such
that, at every point $p\in Y$, one has $\phi  _{p}=r_{p}^{\ast }\left(
\phi _{0}\right) $ for some frame $r_{p}:T_{p}Y\rightarrow \mathbb{R}^{7}$ and (with the sign conventions of \cite{Salamon1989})  
\begin{equation}
\label{eq: G2 3-form}
         \phi_0 
        =e^{567}+\omega_1\wedge e^{5}
        +\omega_2\wedge e^{6}
        +\omega_3\wedge e^{7}
\end{equation}
with
\begin{displaymath}
        \omega_1= e^{12} - e^{34}, \quad 
        \omega_2= e^{13} - e^{42}, \quad 
        \text{and}\quad
        \omega_3= e^{14} - e^{23}.
\end{displaymath}
Moreover, $\phi$ determines a Riemannian metric $g(\phi)$ induced by the pointwise inner-product
\begin{equation}        \label{eq: phi0 gives inner product}
        \left\langle u,v\right\rangle e^{1...7}
        :=-\frac{1}{6}
        \left( u\lrcorner \phi_{0}\right) 
        \wedge 
        \left( v\lrcorner \phi _{0}\right) 
        \wedge 
        \phi _{0}.
\end{equation}
under which $\ast_\phi\phi$ is given pointwise by
\begin{equation}
\label{eq: G2 4-form}
        \ast \phi_0 
        =
        e^{1234}-\omega_1\wedge e^{67}
        -\omega_2\wedge e^{75}
        -\omega_3\wedge e^{56}
        .
\end{equation}
Such a pair $\left( Y,\phi \right) $ is a  \emph{$\rG_2-$manifold} if $\rd\phi =0$ and $\rd\ast _{\phi }\phi =0$. Notice that the co-closed condition is nonlinear in $\phi$, since the Hodge star depends on the metric and hence on $\phi$ itself. 
\subsection{Gauge theory on $\rG_2$-manifolds}

The $\rG_2-$structure allows for a $7-$dimensional analogue of conventional
Yang-Mills theory, yielding a notion analogous to (anti-)self-duality for $2-$forms.
Working in $\R^7$ under the usual identification between $2-$forms and matrices, we have $\mathfrak{g}%
_{2}\subset \mathfrak{so}\left( 7\right) \simeq \Lambda ^{2}$, so we define  $\Lambda _{14}^{2}:= \mathfrak{g}_{2}$
and $\Lambda _{7}^{2}$ its orthogonal complement in $\Lambda ^{2}$:%
\begin{equation}        \label{eq:split}
        \Lambda ^{2}
        =
        \Lambda _{7}^{2}\oplus \Lambda _{14}^{2}.  
\end{equation}%
It is easy to check that $\Lambda _{7}^{2}=\left\langle e_1\lrcorner\phi_0,\dots,e_7\lrcorner\phi_0 \right\rangle$, hence the orthogonal projection onto $\Lambda _{7}^{2}$ in (\ref{eq:split}) is given by 
\begin{eqnarray*}        \label{eq Lphi0}
        L_{\ast \phi _{0}} : \Lambda ^{2} 
        &\rightarrow& \Lambda ^{6}\\ 
        \eta &\mapsto& \eta \wedge \ast \phi _{0}
\end{eqnarray*}
in the sense that
\cite[p. 541]{Bryant1985}
\begin{equation}        \label{prop +-projection}
        L_{\ast \phi _{0}} | _{\Lambda_{7}^{2}}
        :\Lambda _{7}^{2} \: 
        \tilde{\rightarrow} \: \Lambda^{6} 
        \quad\text{and}\quad
        L_{\ast \phi _{0}}|_{\Lambda_{14}^{2}}=0.
\end{equation}
Furthermore, since (\ref{eq:split}) splits $\Lambda^2$ into irreducible representations of $\rG_2$, a little inspection on generators reveals that $\left( \Lambda^{2}\right)_{^{\;7}_{14} }$ is respectively the $_{+1}^{-2}-$eigenspace of the  $%
\rG_{2}-$equivariant linear map%
        \begin{eqnarray*}
        T_{\phi_0} \;: \; \Lambda ^{2} &\rightarrow& \Lambda ^{2} \\
        \eta &\mapsto &T_{\phi_0}\eta :=\ast \left( \eta \wedge \phi _{0}\right) .
\end{eqnarray*}

\subsubsection{Yang-Mills formalism on $\rG_2$-manifolds}
Consider now a $G$--bundle $E\rightarrow Y$ over a compact $\rG_2-$manifold $\left(
Y,\phi \right) $; the curvature $F:=F_{A} $ of some connection $A$ decomposes according to
the splitting (\ref{eq:split}):%
\begin{displaymath}
        F_A=F_{7}\oplus F_{14},
        \qquad  F_{i}\in \Omega_{i}^{2}(Y,\mathfrak{g}_{E}), \;i=7,14,
\end{displaymath}%
where $\mathfrak{g}_{E}$ denotes the adjoint bundle associated to $E$. The $L^2-$norm of $F_{A}$ is the \emph{Yang-Mills functional}, which therefore has two corresponding components:%
\begin{equation}        \label{YM(A)}
        \sY \left( A\right) 
        := \Vert F_{A}\Vert ^{2}
        =\Vert F_7\Vert ^{2}
        +\Vert F_{14}\Vert ^{2}.
\end{equation}

It is well-known that the values of $\sY\left( A\right) $ can be related to a certain
characteristic class of the bundle $E$, given (up to choice of orientation) by%
\begin{displaymath}
        \kappa \left( E\right) :=-\int_{Y}%
        \tr\left( F_{A}^{2}\right) \wedge \phi .
\end{displaymath}%
Using the property $d\phi =0$, a standard argument of Chern-Weil theory \cite{Milnor1974} shows that the de Rham class $\left[ \tr\left(
F_{A}^{2}\right) \wedge \phi \right]$ is independent of $A$, thus the
integral is indeed a topological invariant. The eigenspace decomposition of $T_{\phi}$  implies (up to a sign)%
\begin{displaymath}
        \kappa \left( E\right)
        =
        -2\left\Vert F_{7}\right\Vert^{2}+\left\Vert F_{14}\right\Vert ^{2},
\end{displaymath}
and combining with (\ref{YM(A)}) we get
\begin{displaymath}
        \sY(A)=-\frac{1}{2}\kappa \left( E\right)+\frac{3}{2}\Vert F_{14}\Vert ^{2} 
        =\kappa (E)+3\left\Vert F_{7}\right\Vert ^{2}.
\end{displaymath}%
Hence $\sY\left( A\right) $ attains its absolute minimum at a
connection whose curvature lies either in 
$\Omega_{7}^{2}(Y,\mathfrak{g}_{E})$ or in 
$\Omega_{14}^{2}(Y,\mathfrak{g}_{E})$. Moreover, since $\sY\geq0$, the sign of  $\kappa(E)$ obstructs the existence of one type or the other, so we fix $\kappa(E)\geq0$ and define \emph{$\rG_2-$instantons} as connections with $ F\in\Omega
_{14}^{2}(Y,\mathfrak{g}_{E})$, i.e., such that $\sY(A)=\kappa(E)$. These are precisely the solutions of the 
$\rG_2-$\emph{instanton equation}:
\begin{equation}
\refstepcounter{equation}\tag{\theequation a}
        \label{eq: G2-intanton equation II}
        F_A\wedge\ast\phi=0
\end{equation}
or, equivalently, 
\begin{equation}
\tag{\theequation b}
        \label{eq: G2-intanton equation}
        F_A-\ast\left(F_A\wedge\phi\right)=0.
\end{equation}
If instead  $\kappa(E)\leq0$, we may still reverse orientation and consider
  $ F\in\Omega
_{14}^{2}(Y,\mathfrak{g}_{E})$, but then the above eigenvalues and energy bounds must
 be adjusted accordingly, which amounts to a change of the $(-)$ sign in (\ref{eq: G2-intanton equation}).

\subsubsection{The Chern-Simons functional $\vartheta$}
\label{Subsect Chern-Simmons}

It was pointed out by Simon Donaldson and Richard Thomas in their seminal article on gauge theory in higher dimensions \cite{Donaldson1998} that, formally, $\rG_2$--instantons are rather similar to flat connections over $3$--manifolds; in particular, they are critical points of a Chern--Simons functional and there is hope that counting them could lead to a enumerative invariant for $\rG_2$--manifolds not unlike the Casson invariant for $3$--manifolds, see \cite{Donaldson2011}*{\S 6} and \cite{Walpuski2013}*{Chapter 6}. Although this interpretation has no immediate bearing on the remainder of this material, let us  briefly review the basic formalism, from a purely motivational perspective. 

Given  a bundle over a compact $3-$manifold, with space of connections $\sA$ and gauge group $\sG$, the \emph{Chern-Simons functional} is a multi-valued real function on the quotient $\sB=\sA/\sG$, with integer periods, whose critical points are precisely the flat
connections \cite[\S2.5]{Donaldson2002}.
Similar theories can be formulated in higher dimensions  in the presence of a suitable closed differential form \cite{Donaldson1998,Thomas1997}; e.g. on a $\rG_2-$manifold $(Y,\phi)$, the coassociative
 $4-$form $%
\psi:=\ast \phi $ allows for the definition of a functional of Chern-Simons type\footnote{in fact only the condition $\rd\psi=0$ is required, so the discussion extends to cases in which the $\rG_2-$structure is not necessarily torsion-free.}. Its `gradient', the Chern-Simons $1$-form, vanishes precisely at the $\rG_2-$instantons, hence it detects the solutions to the Yang-Mills equation  \cite{Donaldson2002}. The explicit case of $\rG_2-$manifolds, which we now describe, was  examined in some detail in \cite{SaEarp2009,SaEarp2014}.

The space $\sA$ of connections on  $E\to Y$  is an affine space modelled on $\Omega ^{1}\left( 
\mathfrak{g}_{E}\right) $ so, fixing a reference connection $A_{0}\in 
\sA$,%
\begin{displaymath}
        \sA=A_{0}+\Omega ^{1}\left( Y,\mathfrak{g}_{E}\right) 
\end{displaymath}%
and, accordingly, vectors at  $A\in\sA$
are 1-forms $a,b,\dots\in T_A\sA\simeq\Omega ^{1}\left( Y,\mathfrak{g}_{E}\right) $ and vector fields are maps $\alpha,\beta,\dots:\sA\to\Omega ^{1}\left( Y,\mathfrak{g}_{E}\right)$. In this notation we define the \emph{Chern-Simons functional} by%
\begin{displaymath}       \index{Chern-Simons!functional}
        \vartheta \left( A\right) :=\tfrac{1}{2}\int_{Y}\tr\left( d_{A_{0}}a\wedge         a +\frac{2}{3}a\wedge a\wedge a\right) \wedge \ast \phi ,
\end{displaymath}
fixing $%
\vartheta \left( A_0\right) =0$. This function is
obtained by integration of the \emph{Chern-Simons $1-$form}%
\begin{equation}        \label{ro ^ phi}
                        \index{Chern-Simons!1-form@$1-$form}%
        \rho \left( \beta \right) _{A}:=\int_{Y}%
        \tr\left( F_{A}\wedge \beta_A\right) \wedge \ast \phi .
\end{equation}

It is straightforward to check that the co-closedness condition   $\rd*\phi =0$ implies that the $1-$form (\ref{ro ^ phi}) is closed, so
the procedure doesn't depend on the path $A\left( t\right) $.
Since $\sA$ is contractible, by the Poincar\'{e} Lemma  $%
\rho $ is the derivative of some function $\vartheta $, and
by Stokes' theorem $\rho $ vanishes along $\mathcal{G-}$orbits  $\im \rd_{A} \simeq
T_{A}\left\{ \mathcal{G}.A\right\} $.
Thus $\rho $ descends to the quotient $\sB$ and so does $\vartheta $,  at least locally. 
Since  $\ast\phi$ is not, in general, an integral class, the set of periods of $\vartheta$ is actually  \emph{dense}; however, as
long as our interest remains in the study of the moduli space $\mathcal{M}=\Crit(\rho)$
of $\rG_2$-instantons, there is no worry, for the gradient $\rho =\rd\vartheta $ is unambiguously defined
on $\sB$.

\subsection{Analysis on manifolds with tubular ends}
\label{sec: mfds w tubular ends}

In order to get some more depth into  the instanton gluing process of Theorem \ref{thm:itcs}, we will need some general results from linear analysis on asymptotically cylindrical  manifolds (cf. Definition \ref{def: ACyl mfds}).

\begin{definition}
\label{def: mfd w tubular ends}
A \emph{manifold with tubular end}    $(M,X, \pi)$ is given by a smooth manifold  $M$ with a distinguished compact submanifold-with-boundary $M_0\subset M$,  a Riemannian manifold $X$, and a diffeomorphism 
$$
\pi\colon M_\infty:=M\setminus M_0\to\mathbb{R}_{+}\times X.
$$
 The complement $M_\infty:=M\setminus M_0$ is called the tubular end, $\pi$ is the \emph{tubular model} and $X$ is the \emph{asymptotic cross-section}.
\footnote{The reader interested in analysis on tubular manifolds will find a thorough and very useful toolbox in \cite{Pacini2012}.
}
\end{definition}

Of course one could in principle consider, analogously, manifolds with any number of tubular ends but, in the context of $\rG_2$-manifolds, the Ricci-flat geometry constrains that number to one:  
 \begin{theorem}[{\cite{Salur2006}*{Theorem 1}}]
 If a connected and orientable manifold   $M$  with $k$ tubular ends admits a Ricci-flat metric, then $k\leq2$. Moreover $k=2$ if, and only if, $M$ is a cylinder. 
 \end{theorem}
\subsubsection{Geometric structures on manifolds with cylindrical end}

On a manifold with tubular end $(M,X, \pi)$, we have the following natural maps on differential forms (which clearly extend to any tensor fields):
\[
\xymatrix{ 
\Omega^{\bullet}(M)\ar[r]^{\res}\ar[rd]&\Omega^{\bullet}(M_\infty) \ar[d]^{\pi_*} \\
& \Omega^{\bullet}({\mathbb{R}_{+}\times X})
}
\]

By slight abuse of notation, given $\sigma_{\infty}\in\Omega^{\bullet}(X)$, we will also denote  by $\sigma_{\infty}$ its pullback to the product under  $\mathbb{R}_{+}\times X\xrightarrow[{}]{p_2} X$. Denoting by $t$ the coordinate function on $\R$, we adopt the following notation for asymptotic behaviour: 
\begin{itemize}
\item $\sigma\squigd \sigma_\infty$, if $\vert \nabla^{k}(\pi_{\ast}\sigma-\sigma_\infty)\vert\leq O(e^{-\delta t})$,  $t\in \mathbb{R}_{+}$, $\forall k\geq 0$, for a given $\delta>0$.
\item $\sigma\squig \sigma_\infty$, if $\exists\delta>0$ such that $\sigma\squigd\sigma_\infty$.
\end{itemize}
Whenever $\sigma\rightsquigarrow\sigma_\infty$,  $\sigma$ is said to be \emph{asymptotically translation-invariant} and
 $\sigma_\infty$ is its  \emph{asymptotic limit}.
 
\begin{definition}
\label{def: ACyl mfds}
A manifold with tubular end $(M,X,\pi)$ is said to be \emph{asymptotically cylindrical (\acyl)} if $M$ is also a Riemannian manifold and its metric $g_M$ is asymptotic to the natural cylindrical metric on the tubular model:  $g_M\squig g_X+\rd t^2$. In this case, we will call the  map $\pi\co M_\infty\to \R_+\times  X$ the \emph{cylindrical model}. 
\end{definition}

Let $E_\infty\to X$ be a Riemannian vector bundle.
By slight abuse of notation we also denote by $E_\infty$ its pullback to $\R_+\times X$.
For $k\in\N_0$, $\alpha\in(0,1)$ and $\delta\in\R$ we define
\begin{equation*}
  \|\cdot\|_{C^{k,\alpha}_\delta} := \|e^{-\delta t}\cdot\|_{C^{k,\alpha}},
\end{equation*}
denoting by $C^{k,\alpha}_\delta(X,E_\infty)$ the respective closure of $C^\infty_0(X,E_\infty)$.
We set $C^\infty_\delta:=\bigcap_{k} C^{k,\alpha}_\delta$.

Similarly, a Riemannian vector bundle $E\to M$ over an \acyl\ manifold $(M,X,\pi)$ is said to be \emph{asymptotic} to $E_\infty\to X$ if there is a bundle isomorphism $\bar\pi\co E|_{M_\infty} \to E_\infty$   covering $\pi$ such that the push-forward of the metric on $E$ is asymptotic to the metric on $E_\infty$ in the $C^\infty_\delta$ tubular norm above (for some $\delta>0$).
Denote by $t\co M\to[1,\infty)$ a smooth positive function which agrees with $t\circ\pi$ on $\pi^{-1}([1,\infty)\times  X)$, and define
\begin{equation*}
  \|\cdot\|_{C^{k,\alpha}_\delta} := \|e^{-\delta t}\cdot\|_{C^{k,\alpha}},
  \quad 
  \delta\in\R,
\end{equation*}
denoting by $C^{k,\alpha}_\delta(M,E)$ the respective closure of $C^\infty_0(M,E)$.

Finally, a connection $A\in \sA(E)$ is said to be \emph{asymptotic} to $A_\infty\in\sA(E_\infty)$ if $(A-\bar\pi^*A_\infty)\squig 0$ (the difference of two connections being a $1$-form). We also denote by $A_\infty$ its pullback to $E_\infty\to\R_+\times X$.

\subsubsection{Asymptotically translation-invariant operators on ACyl manifolds}

Let us briefly review some spectral theory for elliptic operators on sections of vector bundles over an \acyl\   manifold $M$ with asymptotic cross-section $X$. The primary references for the material in this section are Maz'ya--Plamenevski{\u\i}~\cite{Mazya1978} and Lockhart--McOwen~\cite{Lockhart1985}.

Let $F\to X$ be a Riemannian vector bundle, and let $D\co C^\infty(X,F)\to C^\infty(X,F)$ be a linear self-adjoint elliptic operator of first order.
The operator
\begin{equation*}
  L_\infty:=\del_t-D
\end{equation*}
extends to a bounded linear operator $L_{\infty,\delta}\co C^{k+1,\alpha}_\delta(X,F)\to C^{k,\alpha}_\delta(X,F)$.

\begin{theorem}[\cite{Mazya1978}*{Theorem 5.1}]
  \label{thm:mazya-plamenevskii}
   $L_{\infty,\delta}$ is invertible if and only if $\delta\notin \spec(D)$.
\end{theorem}

Indeed, elements $a\in\ker L_\infty$ can be expanded in terms of the  $\delta$--eigensections of $D$, see~\cite{Donaldson2002}*{\S~3.1}:
\begin{equation}
  \label{eq:expansion}
  a=\sum_{\delta\in\spec D} e^{\delta t} a_\delta.
\end{equation}

Now let $E\to M$ be a (Riemannian) vector bundle asymptotic to $F$ and consider an elliptic operator 
$$
L\co C_0^\infty(M,E)\to C_0^\infty(M,E)
$$   
asymptotic to $L_\infty$, that is, such that the coefficients of  $L$ are asymptotic to the coefficients of $L_\infty$.
The operator $L$ extends to a bounded linear operator $L_\delta\co C^{k+1,\alpha}_{\delta}(M,E)\to C^{k,\alpha}_{\delta}(M,E)$.

\begin{prop}[\cite{Haskins2015}*{Proposition 2.4}]
  \label{prop:non-critical-fredholm}
  If $\delta\notin\spec(D)$, then $L_\delta$ is Fredholm.
\end{prop}

Elements in the kernel of $L$ still have an asymptotic expansion analogous to \eqref{eq:expansion}.
We need the following result which extracts the constant term of this expansion.

\begin{prop}[{\cite[Proposition 3.5]{SaEarp2015b}}]
  \label{prop:abstract-iota}
  There is a constant $\delta_0>0$ such that, for all $\delta\in[0,\delta_0]$, one has $\ker L_\delta = \ker L_0$ and there is a linear map $\iota\co \ker L_0 \to \ker D$ such that
  \begin{equation*}
    a\overset{\delta_0}{\squig}\iota(a).
  \end{equation*}
  In particular,
  \begin{equation*}
    \ker\iota = \ker L_{-\delta_0}.
  \end{equation*}
\end{prop}

\subsection{Twisted connected sums}
\label{sec:tcs}

An important method to produce examples of compact  $7$--manifolds with holonomy
exactly $\rm G_2$ is the \emph{twisted connected sum} (TCS) construction 
\cite{Kovalev2003,Corti2013,Corti2015}.
It consists of gluing a pair of asymptotically cylindrical (\acyl) Calabi--Yau $3$--folds obtained from certain  smooth projective $3$--folds called  \emph{building blocks} (see Definition \ref{def:matching-data}). 
Combining results of Kovalev and Haskins--Hein--Nordström, each matching pair of building blocks yields a one-parameter family of closed $\rG_2$--manifolds.

 A building block $(Z,\kd)$
 is given by a projective morphism
$\zeta: Z\to \P^1$ such that $\kd:=\zeta^{-1}(\infty)$ is a smooth anticanonical  $K3$
surface, under some mild topological assumptions (see Definition \ref{def:building-block}); in particular, $\kd$ has
trivial normal bundle. Choosing a convenient K\"ahler structure on
$Z$, one can make $V:=Z\setminus \kd$ into an \acyl\ Calabi--Yau $3$--fold (cf. Definition \ref{def: ACylCY3}), that is, a non-compact Calabi--Yau manifold with a tubular end modelled on $\mathbb{R}_+\times \mathbb{S}^1\times \kd$ \cite[Theorem 3.4]{Corti2015}. Then $\mathbb{S}^1\times V$ is an \acyl\ $\rm G_2$--manifold (cf. Definition \ref{def: ACyl G2}) with a tubular end modelled
on $\mathbb{R}_+\times \mathbb{T}^2\times \kd$. 

\begin{definition}[cf. {\cite[Definition 3.9]{Corti2015}}]      
  \label{def:matching-data}
  \label{def:match}
  Let $Z_\pm$ be complex $3$-folds, $\kd_\pm \subset Z_\pm$ smooth anticanonical
$K3$ divisors and $\kclass_\pm \in H^2(Z_\pm)$ K\"ahler classes.
We call a \emph{matching} of $(Z_+, \kd_+, \kclass_+)$
and $(Z_-, \kd_-, \kclass_-)$ a diffeomorphism $\hkr \colon {\kd_+ \to \kd_-}$ such that
$\hkr^* \kclass_- \in H^2(\kd_+)$ and $(\hkr^{-1})^* \kclass_+ \in H^2(\kd_-)$
have type $(2,0) + (0,2)$.

  Given a pair of building blocks $(Z_\pm,\Sigma_\pm)$, a set of  \emph{matching
data} is a collection 
  $
  \bm=\{\(\omega_{I,\pm},\omega_{J,\pm},\omega_{K,\pm}\),\fr\}
  $ consisting of a choice of \hk structures on $\Sigma_\pm$ such
that $[\omega_{I,\pm}]=\kclass_\pm|_{\kd_\pm}$ is the restriction of a Kähler class on $Z_\pm$ and a matching
 $\fr\co\Sigma_+\to\Sigma_-$ such that
  \begin{equation*}
    \fr^*\omega_{I,-}=\omega_{J,+}, \quad
    \fr^*\omega_{J,-}=\omega_{I,+} \qandq
    \fr^*\omega_{K,-}=-\omega_{K,+}.
  \end{equation*} In this case $(Z_\pm,\Sigma_\pm)$ are said to \emph{match} via $\bm$ and $\hkr$ is called  a \emph{\hk rotation} (see Remark \ref{rem: hK ok} below).
\end{definition}

Identifying a matching pair $(Z_\pm,\kd_\pm)$ of building blocks by the \hk rotation between the $K3$ surfaces `at infinity', the corresponding pair $\mathbb{S}^1 \times V_\pm$ of  \acyl\ $\rm G_2$--manifolds is truncated at a
large `neck length' $T$ and, intertwining the circle components in the tori $\mathbb{T}^2_\pm$ along the tubular end, glued to form a compact $7$-manifold
$$
Y =Z_+\#_\hkr Z_-:=\left(  \mathbb{S}^1 \times V_+\right) \cup_{\hkr}  \left(\mathbb{S}^1 \times V_-\right).
$$
For large enough $T_0$, this twisted connected sum $Y$ carries a family of  $\rm G_2$-structures $\{\phi_T\}_{T\geq T_0}$ with $\Hol(\phi_T)=\rm G_2$
\cite[Theorem 3.12]{Corti2015}.
The construction is summarised in the following statement.

\begin{theorem}[{\cite[Corollary 6.4]{Corti2015}}]
\label{thm:tcs}
Given a matching pair
of building blocks $(Z_\pm, \kd_\pm)$ with K\"ahler classes
$\kclass_\pm \in H^{1,1}(Z_\pm)$ such that
$(\kclass_+|_{\kd_+})^2 = (\kclass_-|_{\kd_-})^2$, there exists
a family of torsion-free $G_{2}$-structures
$\left\{\phi_{T}:T\gg1\right\}$ on the closed
$7$-manifold $Y = Z_+\#_\hkr Z_-$.
\end{theorem}

\begin{figure}[h ]
  \centering
  \begin{tikzpicture}
    \coordinate (l) at (0,0);
    \coordinate (r) at (4,0);
    \draw (l) to [out=180,in=0] +(-2,0) to [out=180,in=0] +(-1.2,.4) to [out=180, in=90] +(-.8,-.8) to [out=-90,in=180] +(.8,.-.8) to [out=0,in=180] +(+1.2,+.4) to [out=0,in=180] +(2,0);
    \draw [densely dotted,color=gray] (l) to +(.4,0);
    \draw [densely dotted,color=gray] (l)+(0,-.8) to +(.4,-.8);
    \draw [densely dotted,color=gray] (l)+(.4,0) arc (90:-90:0.1 and 0.4);
    \draw [dotted,color=gray] (l)+(.4,0) arc (90:270:0.1 and 0.4);
    \draw (l) arc (90:-90:0.1 and 0.4);
    \draw (l) [dashed,color=gray] arc (90:270:0.1 and 0.4);
    \draw (l)+(-.3,0) arc (90:-90:0.1 and 0.4);
    \draw (l)+(-.3,0) [dashed,color=gray] arc (90:270:0.1 and 0.4);
    \draw (l)+(-2.2,-1.2) node [right] {$\times$};
    \draw (l)+(-2.2,-1.6) node [right] {$ \mathbb{S}^1$};
    \draw (l)+(-3.1,-.4) node {$V_+$};
    \draw [decorate,decoration={brace,amplitude=3pt}] ($(l)+(-4,.6)$) to node [above=1mm] {$Y_{T,+}$} ($(l)+(0,.6)$);
    \draw (l)+(-.15,0) node [above] {\tiny $[T\!,\!T\!\!+\!\!1]$};

    \draw (r) to [out=0,in=180] +(2,0) to [out=0,in=180] +(1.2,.4) to [out=0,in=90] +(.8,-.8) to [out=-90,in=0] +(-.8,.-.8) to [out=180,in=0] +(-1.2,+.4) to [out=180,in=0] +(-2,0);
    \draw [densely dotted,color=gray] (r) to +(-.4,0);
    \draw [densely dotted,color=gray] (r)+(0,-.8) to +(-.4,-.8);
    \draw [densely dotted,color=gray] (r)+(-.4,0) arc (90:-270:0.1 and 0.4);
    \draw (r) arc (90:-270:0.1 and 0.4);
    \draw (r)+(.3,0) arc (90:-90:0.1 and 0.4);
    \draw (r)+(.3,0) [dashed,color=gray] arc (90:270:0.1 and 0.4);
    \draw (r)+(1.7,-1.2) node [right] {$\times$};
    \draw (r)+(1.7,-1.6) node [right] {$ \mathbb{S}^1$};
    \draw (r)+(+3.1,-.4) node {$V_-$};
    \draw [decorate,decoration={brace,amplitude=3pt}] ($(r)+(0,.6)$) to node [above=1mm] {$Y_{T,-}$} ($(r)+(4,.6)$);
    
    \draw (l)+(.6,0) node [right] {$\Sigma_+$};
    \draw (l)+(.6,-.4) node [right] {$\times$};
    \draw (l)+(.6,-.8) node [right] {$ \mathbb{S}^1$};
    \draw (r)+(-1.3,0) node [right] {$\Sigma_-$};
    \draw (r)+(-1.3,-.4) node [right] {$\times$};
    \draw (r)+(-1.3,-.8) node [right] {$ \mathbb{S}^1$};
    
    \draw [-stealth] ($(l)+(1.3,0)$) to node [above] {\scriptsize $\fr$} ($(r)+(-1.3,0)$);
    \draw [-stealth] ($(l)+(-1.5,-1.6)$) to ($(l)+(1.3,-1.6)$) to [out=0,in=180] ($(r)+(-1.3,-.8)$);
    \draw [color=white,line width=2pt] ($(l)+(1.3,-.8)$) to [out=0,in=180] ($(r)+(-1.3,-1.6)$) to ($(r)+(1.7,-1.6)$) ;
    \draw [-stealth] ($(l)+(1.3,-.8)$) to [out=0,in=180] ($(r)+(-1.3,-1.6)$) to ($(r)+(1.7,-1.6)$);
  \end{tikzpicture}

  \caption{The twisted connected sum of a matching pair of building blocks.}
\end{figure}
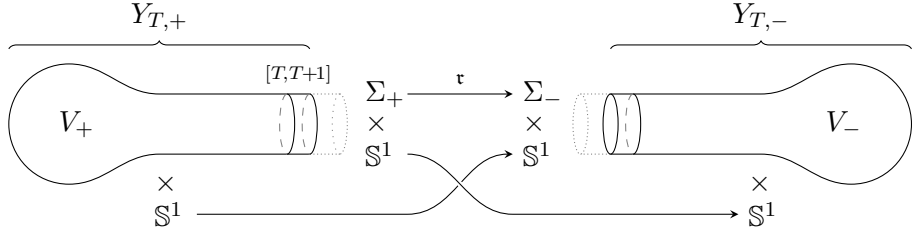

\subsubsection{ACyl Calabi--Yau \texorpdfstring{$3$}{3}--folds from building blocks}

The twisted connected sum in Theorem \ref{thm:tcs} is based on gluing \acyl\  $\rG_2$--manifolds, which arise as the product of an \acyl\ Calabi-Yau $3$-fold with $ \mathbb{S}^1$.
Let us review how to produce these from  building blocks.
\begin{definition}
\label{def: ACylCY3}
  Let   $(V,\omega,\Omega)$ be a Calabi--Yau $3$--fold with tubular end and asymptotic cross-section $\Sigma\times  \mathbb{S}^1$ given by a \hk surface  $(\Sigma,\omega_I,\omega_J,\omega_K)$.
Then $V$    is called an \emph{asymptotically cylindrical Calabi-Yau $3$-fold (\acylcy) } if 
  \begin{align*}
    \omega&\rightsquigarrow\rd t\wedge\rd \alpha + \omega_I , \\
    \Omega&\rightsquigarrow(\rd \alpha-i\rd t)\wedge(\omega_J+i\omega_K),
  \end{align*}
where $t$ and $\alpha$ denote the respective coordinates on $\R_+$ and $ \mathbb{S}^1$.
\end{definition}

Numerous examples of \acylcy\   can be obtained from the following ingredients:

\begin{definition}[Corti--Haskins--Nordström--Pacini \cite{Corti2013}*{Definition~5.1}]
  \label{def:building-block}
  A \emph{building block} is a smooth projective $3$--fold $Z$ together with a projective morphism $\zeta\co Z\to \P^1$ such that the following hold:
  \begin{itemize}
  \item 
  The anticanonical class $-K_Z\in H^2(Z)$ is primitive.
  
  \item 
  $\Sigma:=\zeta^{-1}(\infty)$ is a smooth $K3$ surface and $\Sigma \sim
    -K_Z$.
  
  \item 
  Identifying $H^{2}(\kd,\Z)$ with the $K3$ lattice (\ie choosing a marking for $\kd$), the embedding  
$$
N:=\im(H^{2}(Z,\Z)\rightarrow H^{2}(\kd,\Z))\into H^2(\Sigma)
$$ 
is primitive.
  
  \item 
  The groups $H^{3}(Z,\Z)$ and $H^{4}(Z,\Z)$ are torsion-free.
  \end{itemize}
\end{definition}

In particular, building blocks are simply-connected \cite[\S 5.1]{Corti2013}. 

\begin{remark}
  The existence of the fibration $\zeta\co Z\to \P^1$ is equivalent to $\Sigma$ having trivial normal bundle.
  This is crucial because it means that $Z\setminus \Sigma$ has a cylindrical end, given by an exponential radial coordinate in a tubular neighbourhood of $\Sigma$.
  The last two conditions in the definition of a building block are not essential; they are meant to facilitate the computation of certain topological invariants.
\end{remark}

\begin{remark}
\label{rem: hK ok}
Given a matching $\hkr$ between a pair of
building blocks $(Z_\pm, \kd_\pm, \kclass_\pm)$, 
one can make the choices in the definition of the \acyl\ Calabi-Yau structure
so that $\hkr$ becomes a \hk rotation (cf. Definition \ref{def:matching-data}) of the induced \hk structures
\cite[Theorem~3.4 \& Proposition 6.2]{Corti2015}.
\end{remark}
In his original work, Kovalev~\cite{Kovalev2003} used building blocks arising from Fano $3$--folds by blowing-up the base-locus of a generic anti-canonical pencil.
This method was extended to the much larger class of semi Fano $3$--folds (a class of weak Fano $3$--folds) by Corti--Haskins--Nordström--Pacini~(see Proposition \ref{FanoBlock} below).
Kovalev--Lee~\cite{Kovalev2011} construct building blocks starting from $K3$ surfaces with non-symplectic involutions, by taking the product with $\P^1$, dividing by $\Z_2$ and blowing up the resulting singularities.
In every instance, one obtains an \acylcy\ by the following theorem:
\begin{theorem}[\cite{Haskins2015}*{Theorem D}]
  \label{thm:hhn}
  Let $(Z,\Sigma)$ be a building block and let $(\omega_I,\omega_J,\omega_K)$ be a \hk structure on $\Sigma$.
  If $[\omega_I]\in H^{1,1}(\Sigma)$ is the restriction of a Kähler class on $Z$, then there is an asymptotically cylindrical Calabi--Yau structure $(\omega,\Omega)$ on $V:=Z\setminus \Sigma$ with asymptotic cross section $(\Sigma,\omega_I,\omega_J,\omega_K)$.
\end{theorem}

\begin{remark}
  This result was first claimed by Kovalev in~\cite{Kovalev2003}*{Theorem~2.4}; see the discussion in~\cite{Haskins2015}*{\S 4.1}.
\end{remark}

\subsubsection{Gluing ACyl \texorpdfstring{$\rG_2$}{G2}--manifolds}
\label{sec:gluing-acyl-g2-manifolds}

We may now describe the gluing of matching pairs of \acyl $\rG_2$--manifolds, obtained from \acylcy\ given by Theorem \ref{thm:hhn}.
  
\begin{definition}
\label{def: ACyl G2}
  Let  $(Y,\phi)$ be a $\rG_2$--manifold with tubular end and asymptotic cross-section given by a compact Calabi--Yau $3$--fold  $(W,\omega,\Omega)$. 
   Then $Y$ is called \emph{asymptotically cylindrical (\acyl)} if 
  \begin{equation*}
    \phi\rightsquigarrow\rd t\wedge\omega+\Re\Omega,
  \end{equation*}
  where $t$ denotes the coordinate on $\R_+$.
\end{definition}

Taking the product of an \acylcy\  $(V,\omega,\Omega)$ with $ \mathbb{S}^1$, with coordinate $\beta$, yields an \acyl $\rG_2$--manifold 
\begin{equation*}
  (Y:= \mathbb{S}^1\times V, \phi:=\rd\beta\wedge\omega+\Re\Omega)
\end{equation*}
with asymptotic cross section
\begin{equation*}
  (W:=\mathbb{T}^2\times \Sigma,\omega:=\rd\alpha\wedge\rd\beta+\omega_K,\Omega:=(\rd\alpha-i\rd\beta)\wedge(\omega_J+i\omega_I)).
\end{equation*}

Let $V_\pm$ be a matching pair of \acylcy\  with asymptotic cross section $\Sigma_\pm$ and suppose that $\fr\co \Sigma_+\to \Sigma_-$ is a \hk rotation.
A pair of \acyl $\rG_2$--manifolds $(Y_\pm,\phi_\pm)$ with asymptotic cross sections $(W_\pm,\omega_\pm,\Omega_\pm)$ as above is said to \emph{match} if there exists a diffeomorphism 
$$
\begin{array}{rcrcl}
q&\co&\mathbb{T}^2\times \Sigma_+&\longrightarrow&\mathbb{T}^2\times \Sigma_-\\
  &&f(\alpha,\beta,x)&:=&(\beta,\alpha,\fr(x)).
\end{array}
$$such that
  \begin{equation*}
    q^*\omega_-=-\omega_+ \quad\text{and}\quad
    q^*\Re\Omega_-=\Re\Omega_+.
  \end{equation*}

\begin{remark}
  If $q$ did not interchange the $ \mathbb{S}^1$--factors, then $Y$ would have infinite fundamental group and, hence, could not carry a metric with holonomy equal to $\rG_2$ \cite{Joyce2000}*{Proposition~10.2.2}.
\end{remark}

Let $(Y_\pm,\phi_\pm)$ be a matching pair of \acyl $\rG_2$--manifolds.
For fixed  $T\geq 1$, define 
$$
\begin{array}{rcrcl}
Q&\co&[T,T+1]\times Z_+&\longrightarrow&[T,T+1]\times Z_-\\
  &&  Q(t,z)&:=&\(2T+1-t,q(z)\)
\end{array}
$$
and denote by $Y_T$ the compact $7$--manifold obtained by gluing $Y_\pm$ together at neck length $T$ via $Q$:
\begin{equation*}
  Y_{T,\pm}:=(Y_0)_\pm\cup_Q \pi_\pm^{-1}\((0,T+1]\times Z_\pm\).
\end{equation*}
Fix a non-decreasing smooth cut-off function $\chi\co\R\to[0,1]$ \label{f:chi} with $\chi(t)=0$ for $t\leq 0$ and $\chi(t)=1$ for $t\geq 1$.
Define a $3$--form $\tilde\phi_T$ on $Y_T$ by
\begin{equation*}
  \tilde\phi_T=\phi_\pm-\rd[(\chi(t-T+1)(\omega_\pm-\pi_\pm^*\omega_{\infty,\pm}))]
\end{equation*}
on $Y_{T,\pm}$.
If $T\gg 1$, then $\tilde\phi_T$ defines a closed $\rG_2$--structure on $Y_T$.
Clearly, all the $Y_T$ for different values of $T$ are diffeomorphic; hence, we often drop the $T$ from the notation.
The $\rG_2$--structure $\tilde\phi_T$ is not torsion-free yet, but can be made so by a small perturbation:

\begin{theorem}[\cite{Kovalev2003}*{Theorem~5.34}]
  \label{thm:kovalev}
  In the above situation there exist a constant $T_0\geq 1$ and, for each $T\geq T_0$, a $2$--form $\eta_T$ on $Y_T$ such that $\phi_T:=\tilde\phi_T+\rd\eta_T$ defines a torsion-free $\rG_2$--structure and for some $\delta>0$
  \begin{equation}
    \label{eq:eta-estimate}
    \|\rd \eta_T\|_{C^{0,\alpha}}=O(e^{-\delta T}).
  \end{equation}
\end{theorem}

In summary, the TCS Theorem \ref{thm:tcs} is established by the following procedure.
For any building block $(Z,\kd)$, the noncompact $3$--fold  
$V := Z \setminus \kd$ admits \acyl\ Ricci-flat K\"ahler metrics
(Theorem \ref{thm:hhn}) hence
an \acylcy\  structure whose asymptotic limit defines a \hk structure on $\kd$. 
Given a matching pair of such Calabi-Yau manifolds $V_\pm$,  one can apply Theorem \ref{thm:kovalev} to
glue $ \mathbb{S}^1 \times V_\pm$ into a closed manifold $Y$ with a $1$-parameter family of torsion-free \gtstr s
 \cite[Theorem 3.12]{Corti2015}.

\section{The  $\rG_2$-instanton gluing theorem}

Let $A$ be an ASD instanton on a $\bP U(n)$-bundle $F$ over a K\"ahler surface $\kd$. The linearisation of the instanton moduli space $\mathcal{M}_{\kd}$ near $A$ is modelled on the kernel of the deformation operator
$$
\mathbb{D_A}:= \rd^{*}_{A}\oplus\rd^{+}_{A}: \Omega^{1}(\kd,\mathfrak{g}_{F})
\to 
(\Omega^{0}\oplus\Omega^{+})(\kd,\mathfrak{g}_{F}).
$$ 
  Let  $\cF$ be the corresponding holomorphic vector bundle
(\cf Donaldson-Kronheimer \cite{Donaldson1990}), and denote by $f$ the Hitchin-Kobayashi isomorphism: 
\begin{equation}
\label{eq: isomorphism f}
f: H^{1}(\kd,\Endr_{0}(\cF))\overset{\sim}{\longrightarrow} H^{1}_{A}:=\ker\mathbb{D_A}.
\end{equation}

\begin{theorem}[{\cite[Theorem 1.2]{SaEarp2015b}}]
  \label{thm:HenriqueThomas}
  \label{thm:itcs}
Let $Z_{\pm}$ ,$\kd_{\pm}$, $\kclass_{\pm}$, $\hkr$, $X$ and $\phi_T$ be
as in
Theorem \ref{thm:tcs}.
Let $\cE_{\pm} \to Z_{\pm}$ be a pair of holomorphic vector bundles such
that the following hold: 
\begin{description}
\item[Asymptotic stability]

$\cE_{\pm}|_{\kd_{\pm}}$ is $\mu$-stable with respect to
$\kclass_{\pm}|_{\kd_{\pm}}$.
Denote the corresponding ASD instanton by $A_{\infty,\pm}$. 

\item[Compatibility]
There exists a bundle isomorphism
$\overline{\hkr}:\cE_{+}|_{\kd_{+}}\rightarrow \cE_{-}|_{\kd_{-}}$ covering the \hk
rotation $\hkr$ such that $\overline{\hkr}^{*} A_{\infty,-}=A_{\infty,+}$.

\item[Inelasticity]
There are no infinitesimal deformations of $\cE_{\pm}$ fixing the restriction to
$\kd_{\pm}$:
\begin{equation}      \label{eq:no-deformations}
  H^{1}(Z_{\pm},\Endr_{0}(\cE_{\pm})(-\kd_{\pm}))=0.
\end{equation}

\item[Transversality]
If $\lambda_{\pm}:=f_\pm\circ\res_\pm :
H^{1}(Z_{\pm},\Endr_{0}(\cE_{\pm}))\rightarrow H^{1}_{A_{\infty,\pm}}$ denotes
the composition of restrictions to $\kd_{\pm}$ with the isomorphism
\eqref{eq: isomorphism f}, then the image of $\lambda_{+}$ and
$\overline{\hkr}^{*}\circ \lambda_{-}$ intersect trivially in
the linear space $H^{1}_{A_{\infty,+}}$:
\begin{equation}      \label{eq:trivial-intersection}
  \im (\lambda_{+})\cap \im (\overline{\hkr}^{*}\circ\lambda_{-})=\left\{0\right\}.
\end{equation}
\end{description}
Then there exists a $U(r)$-bundle $E$ over $Y$ and a family of connections $\left\{A_{T}\ :\ T\gg1\right\}$ on the associated
$\bP U(r)$-bundle, such that each $A_{T}$ is an irreducible unobstructed
$G_{2}$-instanton over $(Y,\phi_{T})$.
\end{theorem}

The asymptotic stability assumption guarantees finite energy of Hermitian bundle metrics on  $\cE_\pm|_{V_\pm}$ (see \cite[\S2.2]{}), which are equivalent to asymptotically translation-invariant HYM connections $A_\pm\squig A_{\infty,\pm}$, under the Chern correspondence (cf. Theorem \ref{thm:saearp}).
The maps $\lambda_+$ and
$\overline{\hkr}^{*}\circ\lambda_{-}$ can be seen geometrically as linearisations of the
natural inclusions of the moduli of asymptotically stable bundles
$\mathcal{M}_{Z_\pm}$ into the moduli of ASD instantons $\mathcal{M}_{\kd_+}$
over the $K3$ surface `at infinity', and we think of $H^{1}_{A_{\infty,+}}$ as a
tangent model of  $\mathcal{M}_{\kd_+}$ near the ASD instanton $A_{\infty,+}$.
Then the transversality condition asks that the actual inclusions intersect
transversally at $A_{\infty,+}\in\mathcal{M}_{\kd_+}$.
That the intersection points are isolated reflects that the resulting
$\rm G_2$-instanton is rigid, since it is unobstructed and the deformation
problem has index $0$.

\begin{remark}
  If $H^1(\Sigma_+,\cEnd_0(\cE_+|_{\Sigma_+}))=\{0\}$, then \eqref{eq:trivial-intersection} is vacuous.
  If, moreover, the topological bundles underlying $\cE_\pm$ are isomorphic, then the existence of $\bar\fr$ is guaranteed by~\cite{Huybrechts1997}*{Theorem~6.1.6}.
\end{remark}

Furthermore, condition \eqref{eq:no-deformations} yields  a short exact sequence, which is self-dual under Serre duality: 
$$
  0 \to H^1(Z_\pm,\cEnd_0(\cE_\pm))
    \to H^1(\Sigma_\pm,\cEnd_0(\cE_\pm|_{\Sigma_\pm})) 
    \to H^2(Z_\pm,\cEnd_0(\cE_\pm)(-\Sigma_\pm)) \to 0.
$$
This implies~\cite{Tyurin2012}*{p.\,176 ff.} that 
\begin{equation*}
  \im \lambda_\pm \subset H^1_{A_{\infty,\pm}} 
\end{equation*}
is a complex Lagrangian subspace with respect to the complex symplectic structure induced by $\Omega_\pm:=\omega_{J,\pm}+i\omega_{K,\pm}$ or, equivalently, Mukai's complex symplectic structure on $H^1(Z_\pm,\cEnd_0(\cE_\pm))$.
Under the assumptions of Theorem~\ref{thm:itcs} the moduli space $\sM_{\Sigma_+}$ of holomorphic bundles over $\Sigma_+$ is smooth near $[\cE_+|_{\Sigma_+}]$ and so are the moduli spaces $\sM_{Z_\pm}$ of holomorphic bundles over $Z_\pm$ near $[\cE_\pm]$.
Locally, $\sM_{Z_\pm}$ embeds as a complex Lagrangian submanifold into $\sM_{\kd_\pm}$.
Since $\fr^*\omega_{K,-}=-\omega_{K,+}$, both $\sM_{Z_+}$ and $\sM_{Z_-}$ can be viewed as Lagrangian submanifolds of $\sM_{\kd_+}$ with respect to the symplectic form induced by $\omega_{K,+}$.  Equation~\eqref{eq:trivial-intersection} asks for these Lagrangian submanifolds to intersect transversely at the point $[\cE_+|_{\Sigma_+}]$.
If one thinks of $\rG_2$--manifolds arising via the twisted connected sum construction as analogues of $3$--manifolds with a fixed Heegaard splitting, then this is much like the geometric picture behind Atiyah--Floer conjecture in dimension three \cite{Atiyah1988}.

In \S \ref{sec: transv gluing}, we will review a constructive method to obtain explicit examples of such instanton gluing in many interesting cases.

\subsection{Hermitian Yang-Mills connections on ACyl Calabi-Yau $3$-folds}

Suppose $(W,\omega,\Omega)$ is Calabi--Yau $3$--fold and $(Y:=\R\times W,\phi:=\rd t\wedge\omega+\Re\Omega)$ is the corresponding cylindrical $\rG_2$--manifold.
In this section we relate translation-invariant $\rG_2$--instantons over $Y$ with Hermitian--Yang--Mills connections over $W$.

\begin{definition}
  Let $(W,\omega)$ be a Kähler manifold and let $E$ be a $\PU(n)$--bundle over $W$.
  A connection $A\in\sA(E)$ on $E$ is \emph{Hermitian--Yang--Mills (HYM) connection} if
  \begin{equation}
    \label{eq:hym}
    F_A^{0,2} = 0 \qandq \Lambda F_A = 0.
  \end{equation}
  Here $\Lambda$ is the dual of the Lefschetz operator $L:=\omega\wedge\cdot$.
\end{definition}

\begin{remark}
  Instead of working with $\PU(n)$--bundles, one can also work with $\U(n)$--bundles and instead of the second part of~\eqref{eq:hym} require that $\Lambda F_A$ be equal to a constant.
  These view points are essentially equivalent.
\end{remark}

\begin{remark}
  By the first part of \eqref{eq:hym} a HYM connection induces a holomorphic structure on $E$.
  If $W$ is compact, then there is a one-to-one correspondence between gauge equivalence classes of HYM connections on $E$ and isomorphism classes of polystable holomorphic bundles $\cE$ whose underlying topological bundle is $E$, see Donaldson~\cite{Donaldson1985} and Uhlenbeck--Yau~\cite{Uhlenbeck1986}.
\end{remark}

On a Calabi--Yau $3$--fold, \eqref{eq:hym} is equivalent to
\begin{equation*}
  F_A\wedge\Im\Omega=0 \qandq F_A\wedge\omega\wedge\omega=0;
\end{equation*}
hence, using $*(\rd t\wedge \omega+\Re\Omega)=\frac12\omega\wedge\omega-\rd t\wedge\Im\Omega$ one easily derives:

\begin{prop}[\cite{SaEarp2015a}*{Proposition 8}]
  \label{prop:HYM-G2-instanton}
  Denote by $\pi_W\co Y\to W$ the canonical projection.
  $A$ is a HYM connection if and only if $\pi_W^*A$ is a $\rG_2$--instanton.
\end{prop}

In general, if $A$ is a $\rG_2$--instanton on a $G$--bundle $E$ over a $\rG_2$--manifold $(Y,\phi)$, then the moduli space $\cM$ of $\rG_2$--instantons near $[A]$, i.e., the space of gauge equivalence classes of $\rG_2$--instantons near $[A]$ is the space of small solutions $(\xi,a)\in\(\Omega^0\oplus\Omega^1\)(Y,\fg_E)$ of the system of equations
\begin{equation*}
  \rd_{A}^*a =0 \qandq
  \rd_{A+a}\xi -*(F_{A+a}\wedge\psi) = 0
\end{equation*}
modulo the action of $\Gamma_A\subset \sG$, the stabiliser of $A$, assuming either that $Y$ is compact or appropriate control over  the growth of $\xi$ and $a$.
The infinitesimal deformation theory of $[A]$ is governed by
that equation's  linearisation operator
\begin{equation}
\label{eq:L}
  L_A:=\begin{pmatrix}
    & \rd_A^* \\
    \rd_A & -*(\psi\wedge\rd_A)
  \end{pmatrix}
  \co
  \left(\Omega^0\oplus\Omega^1\right)(Y,\fg_E)\to\left(\Omega^0\oplus\Omega^1\right)(Y,\fg_E).
\end{equation}

\begin{definition}
  \label{def:irreducible-unobstructed}
  $A$ is called \emph{irreducible and unobstructed} if $L_A$ is surjective.
\end{definition}

If $A$ is irreducible and unobstructed, then $\cM$ is smooth at $[A]$.
If $Y$ is compact, then $L_A$ has index zero; hence, is surjective if, and only if, it is invertible; therefore, irreducible and unobstructed $\rG_2$--instantons form isolated points in $\cM$.
If $Y$ is non-compact, the precise meaning of $\cM$ and $L_A$ depends on the growth assumptions made on $\xi$ and $a$; in particular,  $\cM$ may very well be positive-dimensional.

\begin{prop}[{\cite[Proposition 3.13]{SaEarp2015b}}]
  \label{prop:L-dt-D}
  If $A$ is HYM connection on a bundle $E$ over a $\rG_2$--manifold $Y:=\R\times W$ as  in Proposition \ref{prop:HYM-G2-instanton}, then the operator $L_{\pi_W^*A}$ defined in \eqref{eq:L} can be written as
$$    
L_{\pi_W^*A}= \tilde I \del_t + D_{A},
\qwithq        
    \tilde I:=
    \begin{pmatrix}
       & -1 & \\
      1 &   & \\
       && I
    \end{pmatrix}
$$
and
  $D_A\co\left(\Omega^0\oplus\Omega^0\oplus\Omega^1\right)(W,\fg_E)
  \to \left(\Omega^0\oplus\Omega^0\oplus\Omega^1\right)(W,\fg_E)$
  defined by
  \begin{equation}
    \label{eq:D}
    D_A:=
    \begin{pmatrix}
      && \rd_A^*\\
      && \Lambda\rd_A\\
      \rd_A & -I\rd_A & -*(\Im\Omega\wedge\rd_A)
    \end{pmatrix}.
  \end{equation}
\end{prop}

\begin{definition}
  Let $A$ be a HYM connection on a $\PU(n)$--bundle $E$ over a Kähler
  manifold $(W,\omega)$.  Set
  \begin{equation*}
    \cH^i_A := \ker\(\delbar_A\oplus\delbar_A^*\co\Omega^{0,i}\(W,\cEnd_0(\cE)\)\to\(\Omega^{0,i+1}\oplus\Omega^{0,i-1}\)\(W,\cEnd_0(\cE)\)\).
  \end{equation*}
  $\cH^0_A$ is called the space of \emph{infinitesimal automorphisms} of $A$ and $\cH^1_A$ is the space of \emph{infinitesimal deformations} of $A$.
\end{definition}

\begin{remark}
  If $W$ is compact, then $\cH^i_A\iso H^i(W,\cEnd_0(\cE))$ where $\cE$ is the holomorphic bundle induced by $A$.
\end{remark}

\begin{prop}[{\cite[Proposition 3.18]{SaEarp2015b}}]
  \label{prop:kerDA-Hi}
  If $(W,\omega,\Omega)$ is a compact Calabi--Yau $3$--fold and $A$ is a HYM connection on a $G$--bundle $E\to W$, then
  \begin{equation*}
    \ker D_A\iso\cH^0_A\oplus\cH^1_A
  \end{equation*}
  where $D_A$ is as in \eqref{eq:D}.
\end{prop}

\subsection{Gluing \texorpdfstring{$\rG_2$}{G2}--instantons over ACyl \texorpdfstring{$\rG_2$}{G2}--manifolds}

\begin{definition}
  \label{def:ag2i}
  Let $(Y,\phi)$ be an \acyl $\rG_2$--manifold and let $A$ be a  $\rG_2$--instanton on a $G$--bundle over $(Y,\phi)$ asymptotic to $A_\infty$.
  For $\delta\in\R$ we set
  \begin{equation*}
    \cT_{A,\delta} := \ker L_{A,\delta} 
    = \left\{ \ua \in \ker L_A :  \ua \squigd 0 \right\}.
  \end{equation*}
  where $\ua=(\xi,a)\in\left(\Omega^0\oplus\Omega^1\right)(Y,\fg_E)$.
  Set $\cT_A:=\cT_{A,0}$.
\end{definition}

\begin{prop}[{\cite[Propositions 3.22, 3.23]{SaEarp2015b}}]
  \label{prop:iota}
  Let $(Y,\phi)$ be an \acyl\ $\rG_2$--manifold and let $A$ be a $\rG_2$--instanton asymptotic to $A_\infty$.
  Then there is a constant $\delta_0>0$ such that for all $\delta\in[0,\delta_0]$, $\cT_{A,\delta}=\cT_A$ and there is a linear map $\iota\co \cT_{A} \to \cH^0_{A_\infty}\oplus \cH^1_{A_\infty}$ such that
$$
\ua\overset{\delta_0}{\squig}\iota(\ua).
$$ 
In particular,
$    \ker\iota=\cT_{A,-\delta_0}.$
\label{prop:lagrangian}
Furthermore, 
  $$
  \dim\im\iota 
  = 
  \frac12 \dim\(\cH^0_{A_\infty}\oplus\cH^1_{A_\infty}\)
  $$ 
  and, if $\cH^0_{A_\infty}=0$, then $\im\iota\subset\cH^1_{A\infty}$ is Lagrangian with respect to the symplectic structure on $\cH^1_{A_\infty}$ induced by $\omega$.
\end{prop}

Assume we are in the situation of Proposition~\ref{prop:lagrangian}; if moreover $\ker\iota=0$ and $\cH^0_{A_\infty}=0$, then one can show that the moduli space $\cM_Y$ of $\rG_2$--instantons near $[A]$ which are asymptotic to some HYM connection is smooth.
Although the moduli space $\cM_W$ of HYM connections near $[A_\infty]$ is not necessarily smooth, formally, it still makes sense to talk about its symplectic structure and view $\cM_Y$ as a Lagrangian submanifold.
The following theorem shows that transverse intersections of a pair of such Lagrangians give rise to $\rG_2$--instantons:
\begin{theorem}[{\cite[Theorem 3.24]{SaEarp2015b}}]
  \label{thm:gluing}
  Let $(Y_\pm,\phi_\pm)$ be a pair of \acyl $\rG_2$--manifolds that match via $f\co W_+\to W_-$.
  Denote by $(Y_T,\phi_T)_{T\geq T_0}$ the resulting family of compact $\rG_2$--manifolds arising from the construction in \S~\ref{sec:gluing-acyl-g2-manifolds}.
  Let $A_\pm$ be a pair of $\rG_2$--instantons on $E_\pm$ over $(Y_\pm,\phi_\pm)$ asymptotic to $A_{\infty,\pm}$.
  Suppose that the following hold:
  \begin{itemize}
  \item There is a bundle isomorphism $\bar f\co E_{\infty,+}\to E_{\infty,-}$  covering $f$ such that $\bar f^*A_{\infty,-}=A_{\infty,+}$,
  \item The maps $\iota_\pm\co \cT_{A_\pm}\to \ker D_{A_{\infty,\pm}}$ constructed in Proposition~\ref{prop:iota} are injective and their images intersect trivially:
    \begin{equation}
      \label{eq:trivial-intersection-2}
     \im\(\iota_+\) \cap \im\(\bar f^*\circ\iota_-\)=\{0\} \subset \cH^0_{A_{\infty,+}}\oplus \cH^1_{A_{\infty,+}}.
    \end{equation}
  \end{itemize}
  Then there exists $T_1\geq T_0$ and for each $T\geq T_1$ there exists an irreducible and unobstructed $\rG_2$--instanton $A_T$ on a $G$--bundle\ $E_T$ over $(Y_T,\phi_T)$.
\end{theorem}

\begin{proof}[Sketch of proof]
  One proceeds in three steps.
  We first produce an approximate $\rG_2$--instanton $\tilde A_T$ by an explicit cut-and-paste procedure.
  This reduces the problem to solving the non-linear partial differential equation
  \begin{equation}
    \label{eq:perturbation}
    \rd_{\tilde A_t}^* a = 0 \qandq
    \rd_{\tilde A_T+a} \xi + *_T(F_{\tilde A_T + a}\wedge \psi_T) = 0.
  \end{equation}
  for $a\in \Omega^1(Y_T,\fg_{E_T})$ and $\xi\in\Omega^0(Y_T,\fg_{E_T})$ where $\psi_T:=*\phi_T$.
  Under the hypotheses of Theorem~\ref{thm:gluing} one can solve the linearisation of \eqref{eq:perturbation} in a uniform fashion.
  The existence of a solution of \eqref{eq:perturbation} then follows from a simple application of Banach's fixed-point theorem.
\end{proof}

\subsection{From holomorphic bundles over building blocks to \texorpdfstring{$\rG_2$}{G2}--instantons over ACyl \texorpdfstring{$\rG_2$}{G2}--manifolds}
\sectionmark{Holomorphic bundles over building blocks and ...}
\label{sec:holomorphic}

We now briefly explain how one may deduce Theorem~\ref{thm:itcs} from Theorem~\ref{thm:gluing}.

Let $(V,\omega,\Omega)$ be an \acylcy\  with asymptotic cross-section $(\Sigma,\omega_I,\omega_J,\omega_K)$. The following theorem can be used to produce examples of HYM connections $A$   on a $\PU(n)$--bundle $E\to V$  asymptotic to an ASD instanton $A_\infty$ on a $\PU(n)$--bundle $E_\infty\to\Sigma$  (here, by a slight additional abuse, we  denote by $E_\infty$ and $A_\infty$ their respective pullbacks to $\R_+\times  \mathbb{S}^1\times\Sigma$).
Hence, by taking the product with $ \mathbb{S}^1$, it yields examples of $\rG_2$--instantons $\pi_V^*A$ asymptotic to $\pi_\Sigma^*A_\infty$ over the \acyl $\rG_2$--manifold $ \mathbb{S}^1\times V$.
Denote the canonical projections in this context by 
$$
\pi_V\co  \mathbb{S}^1\times V \to V
\qandq
\pi_\Sigma \co \mathbb{T}^2\times \Sigma\to \Sigma.
$$

\begin{theorem}[{\cite{SaEarp2015a}*{Theorem~58} \& \cite{Jacob2016}*{Theorem 1.1}}]
  \label{thm:saearp}
  Let $Z$ and $\Sigma$ be as in Theorem~\ref{thm:hhn} and let $(V:=Z\setminus \Sigma,\omega,\Omega)$ be the resulting \acylcy.
  Let $\cE$ be a holomorphic vector bundle over $Z$ and let $A_\infty$ be an ASD instanton on $\cE|_\Sigma$ compatible with the holomorphic structure.
  Then there exists a HYM connection $A$ on $\cE|_V$ which is compatible with the holomorphic structure on $\cE|_V$ and asymptotic to $A_\infty$.
\end{theorem}

\begin{remark}
The last assertion of the exponential decay $A\squig A_\infty$ is claimed in \cite{SaEarp2015a}*{Theorem~58} but its proof in that reference is not  satisfactory. That part of the theorem is  essentially superseded by \cite{Jacob2016}*{Theorem 1.1},  which additionally extends this existence result to \emph{singular} $\rG_2$-instantons, obtained from asymptotically stable reflexive sheaves, following in spirit the argument from \cite{Bando1994} for the compact case.
\end{remark}

This together with Theorem~\ref{thm:gluing} and the following result immediately implies Theorem~\ref{thm:itcs}.

\begin{prop}[{\cite[Proposition 4.3]{SaEarp2015b}}]
  \label{prop:diagram}
  In the situation of Theorem~\ref{thm:saearp}, 
  suppose $H^0(\Sigma,\cEnd_0(\cE|_\Sigma))=0$.
  Then
  \begin{equation}
    \label{eq:H1-holomorphic-ASD}
    \cH^1_{\pi_\Sigma^*A_\infty} = H^1_{A_\infty}
  \end{equation}
  and, for some small $\delta>0$, there exist injective linear maps 
  $  \kappa_-$ and $\kappa$  such that the following diagram commutes:
  \begin{eqnarray}
  \label{eq:diagram}&&
    \begin{tikzpicture}[baseline=(current  bounding  box.center)]
      \matrix (m) [matrix of math nodes,row sep=3em,column
      sep=2em,minimum width=2em ] {
        \cT_{\pi_V^*A,-\delta} & \cT_{\pi_V^*A} & \cH^1_{\pi_\Sigma^*A_\infty} \\
        H^1(Z,\cEnd_0(\cE)(-\Sigma)) &H^1(Z,\cEnd_0(\cE)) &
        H^1(\Sigma,\cEnd_0(\cE|_\Sigma)). \\}; \path[-stealth]
      (m-1-1) edge 
      (m-1-2) edge node [right] {$\kappa_{-}$} (m-2-1)
      (m-1-2) edge node [above] {$\iota$} (m-1-3)
              edge node [right] {$\kappa$} (m-2-2)
      (m-1-3) edge node [right] {$\iso$} (m-2-3)
      (m-2-1) edge (m-2-2)
      (m-2-2) edge (m-2-3);
    \end{tikzpicture}
  \end{eqnarray}
\begin{proof}[Sketch of proof]

Equation \eqref{eq:H1-holomorphic-ASD} is a direct consequence of $\cH^0_{A_\infty}=0$.
  If $A$ is a HYM connection asymptotic to $A_\infty$ over an \acylcy\  then there exists a $\delta_0>0$ such that, for all $\delta\leq \delta_0$,
  \begin{equation}
    \label{eq:T-kerD}
    \cT_{\pi_V^*A,\delta} =\left\{ \ua \in \ker D_A : \ua\squigd 0 \right\}
  \end{equation}
  with $D_A$ as in \eqref{eq:D}.
  Furthermore, there exists $\delta_1>0$ such that, for all $\delta\leq\delta_1$, one has
  $\cH^0_{A,\delta}=0$ and
  \begin{equation*}
    \cT_{\pi_V^*A,\delta}\iso \cH^1_{A,\delta}
  \end{equation*}
  where $\cH^i_{A,\delta}:=\left\{ \alpha\in\cH^i_A : \alpha\squigd0 \right\}$.
\end{proof}  
\end{prop}

\section{Transversal examples via the  Hartshorne-Serre correspondence}
\label{sec: transv gluing}

In \cite{Kovalev2003,Corti2013,Corti2015}, building blocks $Z$ are produced by blowing up
Fano or semi-Fano 3-folds along the base curve $\sC$ of an anticanonical pencil
(see Proposition \ref{FanoBlock}). By
understanding the deformation theory of pairs $(X,\kd)$ of semi-Fanos $X$
and anticanonical $K3$ divisors $\kd \subset X$, one can produce hundreds of thousands of pairs with the required matching (see \S \ref{subsec:match}).
In order to apply Theorem \ref{thm:itcs} to produce $\rm G_2$-instantons over the
resulting twisted connected sums, one first requires some supply of
asymptotically stable, inelastic vector bundles $\cE \to X$. Moreover, to satisfy
the hypotheses of compatibility and transversality, one would in general need
some understanding of the deformation theory of triples $(X, \kd, \cE)$.
In this Section I present a summary of our approach in \cite{Menet2017} to address this problem of production of ingredients, in the form of gluable pairs of holomorphic bundles over building blocks. 

The Hartshorne-Serre construction generalises the correspondence
between divisors and line bundles, under certain conditions, in the sense
that bundles of higher
rank are associated to subschemes of higher codimension. We recall
the rank $2$ version, as an instance of Arrondo's formulation\footnote{For a thorough justification of this choice of reference for the correspondence, see the Introduction section of Arrondo's notes. }:
\begin{theorem}[{\cite[Theorem
1]{Arrondo2007}}]\label{thm: Hartshorne-Serre}
Let $\sS\subset Z$ be a local complete intersection subscheme of codimension
2 in a smooth algebraic variety. If there exists a line bundle $\mathcal{L}\to
Z$ such that
\begin{itemize}
        \item
        $H^2(Z,\mathcal{L}^*)=0$,
        
        \item
        $\wedge^2\mathcal{N}_{\sS/Z}=\mathcal{L}|_\sS$, where $\mathcal{N}_{\sS/Z}$ denotes the normal bundle of $\sS$ in $Z$.
\end{itemize}
then there exists a rank $2$ holomorphic vector bundle $\cF\to Z$
such that
\begin{enumerate}
        \item
        $\wedge^2\cF=\mathcal{L}$,
        \item
        $\cF$ has one global section whose vanishing locus is $\sS$.
\end{enumerate}
We will refer to such  $\cF$ as the \emph{Hartshorne-Serre bundle obtained
from $\sS$ (and $\mathcal{L}$)}.
\end{theorem}

Using the Hartshorne-Serre construction, we can systematically produce families of bundles over the
building blocks, which, in favourable cases, are parametrised by the building block's blow-up curve $\sC$ itself. This perspective lets us understand the deformation theory of the bundles very explicitly, and it also separates the latter from the deformation theory of the pair $(X,\kd)$.
We can therefore \emph{first} find matchings between two semi-Fano families using the
techniques from \cite{Corti2015}, and \emph{then} exploit the high degree of
freedom in the choice of the blow-up curve $\sC$ (see Lemma \ref{lem:wiggle})
to satisfy the compatibility and transversality hypotheses.

\subsection{A detailed example}

As a proof of concept, we will henceforth walk through the process of construction of examples, with the particular
pair adopted in  \cite{Menet2017}: 

\begin{example}
\label{ex:2conics}
The product $X_+=\P^1\times\P^2$ is a Fano 3-fold. 
Let $\left|\kd_{0},\kd_{\infty}\right|\subset \left|-K_{X_+}\right|$ be a generic
pencil with (smooth) base locus $\sC_+$ and
$\kd_+\in \left|\kd_{0},\kd_{\infty}\right|$ generic.
Denote by $r_+:Z_+\rightarrow X_+$  the blow-up of $X_+$ in $\sC_+$,  by 
 $\wt{\sC_+}$ 
the exceptional divisor and by $\ell_+$ a fibre
of $p_{1}:\wt{\sC_+}\rightarrow \mathscr{C}_+$. The proper transform of $\kd_+$ in $Z_+$ is also denoted by $\kd_+$, and
$(Z_+,S_+)$ is a building block by Proposition \ref{FanoBlock}.
For future reference, we fix classes 
$$
H_+:=r^{*}_+(\left[\P^1\times\P^1\right])
\qandq
G_+=r^{*}_+(\left[\left\{x\right\}\times\P^2\right]) \in H^2(Z_+).
$$

\noindent \textbf{NB.:} Clearly $-K_{X_+}$ is very ample, thus also  $-K_{X_+|\kd_+}$, so $X_+$ lends itself to application of Lemma \ref{lem:wiggle}. 
\end{example}

\begin{example}
\label{ex:2conics2}
A double cover $\pi:X_-\overset{2:1}{\longrightarrow}
\P^1\times\P^2$ branched over a smooth $(2,2)$ divisor $D$ is
a Fano 3-fold.
Let $\left|\kd_{0},\kd_{\infty}\right|\subset \left|-K_{X_-}\right|$ be a
generic pencil with (smooth) base locus $\sC_-$ and
$\kd_-\in \left|\kd_{0},\kd_{\infty}\right|$ generic.
Denote by $r_-:Z_-\rightarrow X_-$ the blow-up of $X_-$ in $\sC_-$, and by  $\wt{\sC_-}$ the exceptional divisor.
The proper transform of $\kd_-$ in $Z_-$ is also denoted by $\kd_-$, and
$(Z_-,S_-)$ is a building block by Proposition \ref{FanoBlock}.
For future reference, we fix classes
\[
H_-:=(r_-\circ \pi)^{*}(\left[\P^1\times\P^1\right])
\qandq
G_-=(r_-\circ \pi)^*(\left[\left\{x\right\}\times\P^2\right])
\in H^2(Z_-) 
\]
 and
\[
h_-:=\frac{1}{2}(r_-\circ \pi)^*(\left[\left\{x\right\}\times\P^1\right])
\in H^4(Z_-),
\]
 where $x$ is a point. 
%
 \end{example}

In that context, the existence of solutions satisfying the hypotheses of the TCS $\rG_2$-instanton gluing theorem takes the following form:

\begin{theorem}[{\cite[Theorem 1.3]{Menet2017}}]
\label{thm:main}
There exists a matching pair of building blocks $(Z_\pm, \kd_\pm)$,
obtained as $Z_\pm=\Bl_{\sC_\pm}X_\pm$ for  $X_+=\P^1\times\P^2$ and the double cover $X_-\overset{2:1}{\longrightarrow}\P^1\times\P^2$ branched over
 a $(2,2)$ divisor, with rank $2$ holomorphic bundles $\cE_\pm \to Z_\pm$ satisfying the
hypotheses of Theorem \ref{thm:itcs}.
\end{theorem}

Here's a sketch of the procedure leading to Theorem \ref{thm:main}:

\begin{itemize}
\item
We construct holomorphic bundles on building blocks from certain complete intersection subschemes, via the Hartshorne-Serre correspondence (Theorem \ref{thm: Hartshorne-Serre}), as
well as two families 
of bundles $\{\cF_{\pm}\to X_\pm \} $, over the particular blocks of Theorem \ref{thm:main}, that are conducive to application of
Theorem \ref{thm:itcs}.

\item
Then, in \S \ref{sec moduli assoc to F|S}, we focus on the moduli space
$\sM_{\kd_+,\mathcal{A}_+}^{s}(v_{\kd_+})$  of stable bundles on $\kd_+$, where the problems of compatibility and transversality therefore ``take place''.
Here $X_+=\P^1\times\P^2$, 
$\kd_+ \subset X_+$ is the anti-canonical $K3$ divisor and, for a smooth curve $\sC_+ \in |{-}K_{X_+|\kd_+}|$, the block
$Z_+ := \textrm{Bl}_{\sC_+} X_+$ is in the family obtained from Example \ref{ex:2conics}.

It can be shown  that $\sM_{\kd_+,\mathcal{A}_+}^{s}(v_{\kd_+})$ is isomorphic to $\kd_+$ itself, and that the restrictions of the family of bundles $\cF_+$
correspond precisely to the blow-up curve $\sC$.
Now, given a rank $2$ bundle $\cF_+\to Z_+$ such that $\mathcal{G}:= \cF_{+|\kd_{+}} \in \sM^{s}_{\kd_{+},\mathcal{A}_{+}}(v_{\kd_+})$,
 the restriction map
\begin{equation}
\label{eq:resintro}
\res : H^1\left(Z_+, \Endr_0(\cF_+)\right) \to H^1(\kd_+, \Endr_0(\mathcal{G}))
\end{equation}
corresponds to the derivative at $\cF_+$ of the map between instanton moduli
spaces. Combining with Lemma \ref{lem:wiggle}, which guarantees the freedom
to choose $\sC_+$ when constructing the block $Z_+$ from $\kd_+$, one has:

\begin{theorem}[{\cite[Theorem 1.4]{Menet2017}}]
\label{thm:summaryintro}
For every
$\mathcal{G} \in  \sM_{\kd_+,\mathcal{A}_+}^{s}(v_{\kd_+})$ and
every line $V \subset H^1(\kd_+, \Endr_0(\mathcal{G}))$, there is
a smooth base locus curve $\sC \in |{-}K_{X_+|\kd_+}|$ and an exceptional fibre $\ell\subset\widetilde{\sC}$ corresponding by Hartshorne-Serre to an  inelastic vector bundle $\cF_+ \to Z_+$, such that $\cF_{+|\kd_+} = \mathcal{G}$ and the  restriction map
\eqref{eq:resintro} has image $V$.
\end{theorem}

\item
Let $\hkr : \kd_+ \to \kd_-$ be a matching between $X_+=\P^1\times\P^2$ and $X_-\overset{2:1}{\longrightarrow}\P^1\times\P^2$.
Then for any $\cF_- \to Z_-$ as above we can (up to a twist by holomorphic line
bundles $\mathcal{R}_\pm\to Z_\pm$) choose the smooth curve
$\sC_+ \in |{-}K_{X_+ | \kd_+}|$ in the construction of $Z_+$ so that there
is a Hartshorne-Serre bundle $\cF_+\to Z_+$ that matches $\cF_-$ transversely.
Then the bundles $\cE_\pm:=\cF_\pm\otimes \mathcal{R}_\pm  $ satisfy
all the gluing hypotheses of Theorem \ref{thm:itcs}.
\end{itemize}

\subsection{Building blocks from semi-Fano 3-folds and twisted connected sums}
\label{sec:Building-Fano}

For all but 2 of the 105 families of Fano $3$-folds, the base locus of a generic
anti-canonical pencil is smooth. This also holds for most families in the
wider class of `semi-Fano $3$-folds' in the terminology of \cite{Corti2013},
\ie smooth projective $3$-folds where $-K_X$ defines a morphism that does not
contract any divisors. 
We can then obtain building blocks using \cite[Proposition 3.15]{Corti2015}:

\begin{prop}\label{FanoBlock}
Let $X$ be a semi-Fano 3-fold with $H^{3}(X,\Z)$ torsion-free, $|\kd_{0},\kd_{\infty}|\subset |-K_{X}|$ a generic pencil with (smooth) base locus $\sC$, $\kd\in |\kd_{0},\kd_{\infty}|$ generic, and $Z$ the blow-up of $X$ at $\mathscr{C}$. Then $\kd$ is a smooth $K3$ surface, its proper transform in $Z$ is isomorphic to $\kd$, and $(Z,\kd)$ is a building block. Furthermore
\begin{enumerate}
\item
the image $N$ of $H^{2}(Z,\Z)\rightarrow H^{2}(\kd,\Z)$ equals that of $H^{2}(X,\Z)\rightarrow H^{2}(\kd,\Z)$;

 \item
 $H^{2}(X,\Z)\rightarrow H^{2}(\kd,\Z)$ is injective and the image $N$ is primitive in $H^{2}(\kd,\Z)$.
\end{enumerate}
\end{prop}

Let us notice for later use that, whenever $-K_X|_\kd$ is very ample, it is possible to `wiggle' a blow-up curve $\sC$ so as to realise any prescribed incidence condition $(x,V)\in T\kd$. 
This fact will play an  important role in the transversality argument
in \S \ref{sec moduli assoc to F|S}.

\begin{lemma}[{\cite[Lemma 2.5]{Menet2017}}]
\label{lem:wiggle}
Let $X$ be a semi-Fano, $\kd \in |{-}K_X|$ a smooth $K3$ divisor, and suppose
that the restriction of $-K_X$ to $\kd$ is very ample.
Then given any point $x \in \kd$ and any
(complex) line $V \subset T_x \kd$, there exists an anticanonical pencil
containing $\kd$ whose base locus $\sC$ is smooth, contains $x$, and
$T_x \sC = V$.
\end{lemma}

Finally, note that if $X_\pm$ is a pair of semi-Fanos and $\hkr : \kd_+ \to \kd_-$
is a matching in the sense of Definition \ref{def:match}, then $\hkr$ also
defines a matching of building blocks constructed from $X_\pm$ using 
Proposition \ref{FanoBlock}. Thus given a pair of matching semi-Fanos we can
apply Theorem \ref{thm:tcs} to construct closed \gtmfd s, \emph{but} this
still involves choosing the blow-up curves $\sC_\pm$.

\subsection{The matching problem}
\label{subsec:match}

We now explain in more detail the argument of \cite[\S 6]{Corti2015} for finding matching building blocks
$(Z_\pm, \kd_\pm)$.
The blocks will be obtained by applying Proposition \ref{FanoBlock} to a pair
of semi-Fanos $X_\pm$, from some given pair of deformation types $\cX_\pm$.

A key deformation invariant of a semi-Fano $X$ is its Picard lattice
$\Pic(X) \cong H^2(X; \Z)$. For any anticanonical $K3$ divisor
$\kd \subset X$, the injection $\Pic(X) \into H^2(\kd;\Z)$ is primitive.
The intersection form on $H^2(\kd;\Z)$ of any $K3$ surface is isometric to
$L_{K3} := 3\rU \oplus 2\rE_8$, the unique even unimodular lattice of
signature $(3,19)$.
We can therefore identify $\Pic(X)$ with a primitive sublattice
$N \subset L_{K3}$ of the $K3$ lattice, uniquely up to the action of the isometry
group $O(L_{K3})$ (this is usually uniquely determined  by the isometry
class of $N$ as an abstract lattice).

Given a matching $\hkr \colon \kd_+ \to \kd_-$ between anticanonical
divisors in a pair of semi-Fanos, we can choose the isomorphisms $H^2(\kd_\pm;\Z) \cong L_{K3}$ compatible with $\hkr^*$, 
hence identify $\Pic(X_+)$ and $\Pic(X_-)$
with a \emph{pair} of primitive sublattices $N_+, N_- \subset L_{K3}$.
While the $O(L_{K3})$ class of $N_\pm$ individually depends only on $X_\pm$,
the $O(L_{K3})$ class of the pair $(N_+, N_-)$ depends on $\hkr$, and we
call $(N_+, N_-)$ the \emph{configuration} of $\hkr$.
Many important properties of the resulting twisted connected sum only depend on the
\hk rotation in terms of  the configuration.  

 Given a configuration $N_+, N_- \subset L_{K3}$, let
\[
N_0 := N_+ \cap N_- , 
\quad\text{and}\quad
\nres_\pm := N_\pm \cap N_\mp^\perp .
\]
We say that the configuration is \emph{orthogonal} if $N_\pm$ are rationally
spanned by $N_0$ and $\nres_\pm$
(geometrically, this means that the reflections in $N_+$ and $N_-$ commute).
Given a pair $\cX_\pm$ of
deformation types of semi-Fanos, then there are sufficient conditions for a given orthogonal configuration
to be realised by some matching \cite[Proposition 6.17]{Corti2015},  
\begin{prop}
\label{6.18}
i.e., so that there exist $X_{\pm}\in \mathcal{X}_{\pm}$,
$\,\kd_{\pm}\in |{-}K_{X_{\pm}}|$, and a matching
$\hkr:\kd_{+}\rightarrow \kd_{-}$ with the given configuration.
\end{prop}

Now consider the problem of finding matching bundles $\cE_\pm \to Z_\pm$ in order
to construct $\rm G_2$-instantons by application of Theorem \ref{thm:HenriqueThomas}.
For the compatibility hypothesis it is obviously necessary that Chern classes match:
\[ c_{1}(\cE_{+}|_{\kd_{+}}) = \hkr^{*}c_{1}(\cE_{-}|_{\kd_{-}}) \in H^2(\kd_+) . \]
Identifying $H^2(\kd_+; \Z) \cong L_{K3} \cong H^2(\kd_-; \Z)$ compatibly with
$\hkr^*$, this means we need
\[ c_1(\cE_{+|\kd_+}) \; = \; c_1(\cE_{-|\kd_-}) \; \in \; N_+ \cap N_-
\; = \; N_0 . \]
Hence, if $N_0$ is trivial,  both $c_1(\cE_{\pm|\kd_\pm})$ 
must also be trivial, which is a very restrictive condition on our bundles.
To allow more possibilities, we want  matchings $\hkr$
whose configuration $N_+, N_- \subset L_{K3}$ has non-trivial intersection
$N_0$.
Table 4 of \cite{Crowley2014} lists all 19 possible such matchings
with Picard rank $2$, among which we can find the pair of building blocks of Examples \ref{ex:2conics} and \ref{ex:2conics2}, coming from the Fano 3-folds  $X_+=\P^1\times\P^2$ and the double cover  $X_-\overset{2:1}{\longrightarrow}\P^1\times\P^2$ branched
over a $(2,2)$ divisor.
Several other choices would be possible to produce examples of $\rm G_2$-instantons.

\subsection{Hartshorne-Serre bundles over building blocks}

\subsubsection{The general construction algorithm}
\label{sec: general algorithm}

Let $X$ be a semi-Fano $3$--fold and $(Z,\kd)$ be the block constructed as a
blow-up of $X$ along the base locus $\sC$ of a generic anti-canonical pencil
(Proposition \ref{FanoBlock}).  
In \cite[\S3.1]{Menet2017} a general approach is provided for making the choices of $\mathcal{L}$ and
$\sS$ in Theorem \ref{thm: Hartshorne-Serre}, in order to construct a Hartshorne-Serre bundle $\cF\to Z$ which, up to a twist, yields the bundle $\cE$
meeting the requirements of Theorem \ref{thm:itcs}.
The approach may be summarised as follows:

\begin{summary}\label{generaltec}
Let $(Z_\pm,\kd_\pm)$ be
the building blocks constructed by blowing-up  $N_\pm$-polarised semi-Fano $3$-folds $X_\pm$ along the base locus $\sC_\pm$ of
a generic anti-canonical pencil (\cf Proposition \ref{FanoBlock}). 
Let $N_{0}\subset N_\pm$ be the sub-lattice of  orthogonal matching, as in \S \ref{subsec:match}.
Let $\mathcal{A}_\pm$ be the restriction of an ample class of $X_\pm$ to $\kd_\pm$ which is orthogonal to $N_{0}$.  We look for the Hartshorne-Serre parameters $\sS_\pm$ and $\mathcal{L}_\pm$ of Theorem \ref{thm: Hartshorne-Serre}, where  $\sS_+=\ell$ is an exceptional fibre in $Z_+$,  $\sS_-$ is a genus $0$ curve in $Z_-$ and  $\mathcal{L}_\pm\to Z_\pm$ are line bundles such that:
\setlength{\columnsep}{-80pt}
\raggedcolumns
\begin{multicols}{2}
\begin{enumerate}
\item
$c_{1}(\mathcal{L}_\pm) \in  N_{0}  \mod 2\Pic (\kd_\pm)$;

\item
$c_{1}(\mathcal{L}_{\pm|\kd_\pm})\cdot \mathcal{A}_\pm>0$;

\item
$c_{1}(\mathcal{L}_+)\cdot \sS_+=-1$ and \\
$(S_--c_1(\mathcal{L}_-))\cdot \sS_-=2$;

\item 
$c_{1}(\mathcal{L}_{+|\kd_+})^{2}=-4$ and $\kd_-\cdot \sS_- -\frac{1}{4}c_{1}(\mathcal{L}_{-|\kd_-})^{2}=2$;

\item
$\chi(\mathcal{L}^*_{+})=0$;

\item
$h^0(\mathcal{N}_{\sS_\pm/{Z_\pm}})=1+h^0(\mathcal{L}_\pm\otimes\mathcal{I}_{\sS_\pm})$
\end{enumerate}
\end{multicols}
\noindent Finally, among candidate data satisfying these constraints, inelasticity must be arranged ``by hand'.  

\end{summary}
The reader who would like to construct other examples might follow this 4-step programme: 
\begin{description}
\item[Step 1.]
Find two matching $N_\pm$-polarized semi-Fano $3$-folds $X_\pm$ such that:
\begin{itemize}
\item[(i)] 
there exists $x\in N_+$ such $x^2=-4$ (or more generally  $x^2=2k-6$, for a moduli space $\sM^{s}_{\kd,\mathcal{A}}(v)$ of dimension $2k$).
\item[(ii)]
there exists a primitive element $y\in N_0$ such that $y^2\leq -8$ and $4$ divides $y^2$.
\end{itemize}
\item[Step 2.]
Find $\mathcal{L}_\pm$ and $\sS_-$ which verify the conditions of Summary \ref{generaltec} (perhaps with  a computer).
\item[Step 3.]
The following must be checked by  \emph{ad-hoc} methods:
        \begin{enumerate}
        \item
         $H^2(\mathcal{L}^*_\pm)=0$, for the Hartshorne-Serre construction (Theorem \ref{thm: Hartshorne-Serre});
        \item
        $H^1(\mathcal{L}^*_{+})=0$, for transversality;        
        \item
        that divisors with small slope do not contain $\sS$, for asymptotic stability \cite[Proposition 10]{Jardim2017};
        \item
        $H^1(\cF_\pm)=0$ for  inelasticity (Proposition \ref{PInelastic2}).        
        \end{enumerate}
\item[Step 4.]
Conclude with similar arguments to \S \ref{sec moduli assoc to F|S}.
\end{description}

\subsubsection{Construction of \texorpdfstring{$\cF_+$}{\cF+} over  
\texorpdfstring{$X_+=\P^1\times\P^2$}
{X+ = P1 x P2} and \texorpdfstring{$\cF_-$}{\cF-} over 
\texorpdfstring{$X_-\protect\overset{2:1}{\longrightarrow}\P^1\times\P^2$}
{X-  -> P1 x P2}
} 
\label{construction}
\label{construction2}

In view of the constraints in Summary \ref{generaltec}, we apply Theorem \ref{thm: Hartshorne-Serre} to $Z_+=\Bl_\sC X_+$ as above, obtained by blowing up $X_+=\P^1\times\P^2$ from  Example
\ref{ex:2conics},
with parameters 
$$
\sS=\ell
\qandq\mathcal{L}=\mathcal{O}_{Z_+}(-\kd_+-G_+ +H_+).
$$
\begin{prop}[{\cite[Propositions 3.5, 4.4, 5.9]{Menet2017}}]\label{prop:exa}
Let $(Z_+,\kd_+)$ be a building block  
as in Example \ref{ex:2conics},
$\sC$ a pencil base locus and $\ell\subset Z_+$ an exceptional fibre of  $\wtilde{\sC}\rightarrow\mathscr{C}$.
There exists a rank $2$ asymptotically stable and inelastic Hartshorne-Serre  bundle $\cF_+\to Z_+$ obtained from $\ell$ such that 

\begin{enumerate}
\item
$c_{1}(\cF_+)=-\kd_+-G_++H_+$, and
\item
$\cF_+$ has a global section whose vanishing locus is a fibre $\ell$  of
$p_1:\wt{\sC}\rightarrow \mathscr{C}$.
\end{enumerate}
\end{prop}

Similarly, one applies Theorem \ref{thm: Hartshorne-Serre} to the
building block $Z_-$ obtained by blowing up  $X_-\overset{2:1}{\longrightarrow}\P^1\times\P^2$, from Example  \ref{ex:2conics2}, with
\[
[\sS]=h_-
\qandq
\mathcal{L}=\mathcal{O}_{Z_-}(G_-).
\]
\begin{prop}[{\cite[Propositions 3.9, 4.5, 5.10]{Menet2017}}]\label{prop:exa2}
Let $(Z_-,\kd_-)$ be a building block 
provided in Example \ref{ex:2conics2} and $\sS$ a line of class $h_-$.
There exists a rank 2 Hartshorne-Serre  bundle $\cF_-\to Z_-$ obtained from $\sS$ such that:
\begin{enumerate}
\item
$c_{1}(\cF_-)=G_-$, and
\item
$\cF_-$ has a global section whose vanishing locus is $\sS$, where $\left[\sS\right]=h_-$.
\end{enumerate}

\end{prop}

\begin{remark}
In order to check the stability of Hartshorne-Serre bundles over  $\kd_\pm$,
we use a tailor-made instance  \cite[Proposition 10]{Jardim2017} of a more
general Hoppe-type stability criterion for holomorphic bundles over so-called  \emph{polycyclic varieties},
whose Picard group is free Abelian  \cite[Corollary 4]{Jardim2017}. That tool allows one to mass-produce examples of holomorphic bundles, over building blocks, which are asymptotically stable, hence admit HYM metrics (\cf Theorem \ref{thm:saearp}).
\end{remark}



In the context above, the moduli spaces of the stable bundles $\cF_{\pm|\kd_\pm}$
have the `minimal' positive dimension, for transversal intersection to occur:
\begin{prop}\label{qui}
Let $(Z_\pm,\kd_\pm)$ be the building block provided in Examples \ref{ex:2conics} and \ref{ex:2conics2}, and let  $\cF_\pm\to Z_\pm$ be the asymptotically stable bundles constructed in
Propositions \ref{prop:exa} and \ref{prop:exa2}.
Let $\sM^{s}_{\kd_\pm,\mathcal{A}_\pm}(v_\pm)$ be the moduli space of $\mathcal{A}_\pm$-stable bundles on $\kd_\pm$ with Mukai vector $v_\pm=v(\cF_{\pm|\kd_\pm})$. We have:
$$\dim \sM^{s}_{\kd_\pm,\mathcal{A}_\pm}(v_\pm)=2.$$

\end{prop}

Recall that (see eg. \cite{Huybrechts2010}) that the \emph{Mukai vector} of  a vector
bundle $\cF\to \kd$ on a $K3$ surface is
defined as
\[
v(\cF):=\left(\rk \cF,\, c_{1}(\cF),\,\chi(\cF)-\rk \cF \right)\in  \left(H^{0}\oplus H^{2}\oplus H^{4}\right)(\kd,\Z),
\] 
with
$\chi(\cF)=\frac{c_{1}(\cF)^2}{2}+2\rk \cF-c_{2}(\cF)$.

\subsubsection{Inelasticity of asymptotically stable Hartshorne-Serre bundles}\label{sec:inelasticity}

These results hold for general building blocks and may be of independent interest.
Recall that a bundle $\cF$ over a building block $(Z,\kd)$ is \emph{inelastic} if 
$$
H^{1}(Z,\Endr_{0}(\cF)(-\kd))=0.
$$
This condition means that there are no global deformations of the bundle $\cF$ which maintain  $\cF_{|\kd}$ fixed at infinity. The following  characterisation of inelasticity,
in the case of asymptotically stable bundles, relates the freedom to extend $\cF$ and the dimension of the moduli space $\sM^{s}_{\kd,\mathcal{A}}(v_\cF)$.
The proof uses Serre duality and Maruyama's characterisation of the moduli space of stable  bundles over a polarised $K3$ surface \cite[Proposition 6.9]{Maruyama1978}.

\begin{prop}\label{Inelasticity1}
Let $(Z,\kd)$ be a building block and $\cF$ an asymptotically stable bundle on $Z$. Let $\sM^{s}_{\kd,\mathcal{A}}(v)$ be the moduli space of $\mathcal{A}$-$\mu$-stable bundles on $\kd$ with Mukai vector $v=v(\cF_{|\kd})$.
Then $\cF$ is inelastic if and only if 
$$\dim \Ext^{1}(\cF,\cF)=\frac{1}{2}\dim \sM^{s}_{\kd,\mathcal{A}}(v).$$
\end{prop}For Hartshorne-Serre bundles of rank $2$ satisfying certain topological hypotheses, we may express the half-dimension of the moduli space in terms of the construction data:
\begin{prop}[{\cite[Corollary 5.8]{Menet2017}}]\label{PInelastic2}
Let $(Z,\kd)$ be a building block, and let $\cF \to Z$ be an asymptotically stable Hartshorne--Serre bundle obtained from a genus $0$ curve $\sS\subset Z$ and a line bundle $\mathcal{L}\to Z$ as in Theorem \ref{thm: Hartshorne-Serre}. 
Let $\sM^{s}_{\kd,\mathcal{A}}(v)$ be the moduli space of $\mathcal{A}$-$\mu$-stable bundles on $\kd$ with Mukai vector $v=v(\cF_{|\kd})$.
Assume moreover that $H^1(\cF)=0$.

Then $\cF$ is inelastic if and only if 
\begin{equation}        \label{eq: dim inelast}
\frac{1}{2}\dim \sM^{s}_{\kd,\mathcal{A}}(v)
=
h^0(\mathcal{N}_{\sS/Z})+h^1(\mathcal{L}^*)-h^0(\cF)+1.
\end{equation}
\end{prop}

\subsection{Proof of Theorem  \ref{thm:summaryintro}}
\label{sec moduli assoc to F|S}

Let $X_+=\P^1\times\P^2$ as in
Example \ref{ex:2conics}, and
$\kd_+ \subset X_+$ be a smooth anti-canonical $K3$ divisor.
For suitable choices of polarisation $\mathcal{A}_+$ on $\kd_+$ and
Mukai vector $v_{\kd_+}$, the associated moduli space $\sM^{s}_{\kd_+,\mathcal{A}_+}(v_{\kd_+})$ of (rank $2$) $\mathcal{A}_+$-stable bundles   is
$2$-dimensional.
For a smooth curve $\sC \in |{-}K_{X_+|\kd_+}|$, let $Z_+ := \textrm{Bl}_{\sC} X_+$
be the building block resulting from Proposition \ref{FanoBlock}. Then, for each exceptional fibre $\ell\subset\widetilde{\sC}$,  the  Mukai vector 
$$
v_{Z_+}:=(2, -\kd_+-G_++H_+, \ell)\in \left(H^{0}\oplus H^{2}\oplus H^{4}\right)(Z_+,\Z)
$$
has the property that, given a bundle $\cF_+\to Z_+$ as in Proposition \ref{prop:exa} with $(\rk \cF_+, c_{1}(\cF_+),c_{2}(\cF_+))=v_{Z_+}$, the restriction
to $\kd_+$ has Mukai vector $v_{\kd_+}$, so $\mathcal{G} := \cF_{+|\kd_+} \in \sM^{s}_{\kd_+,\mathcal{A}_+}(v_{\kd_+})$.
Thus the Hartshorne-Serre construction yields a
family of asymptotically stable vector bundles $\{(\cF_+)_p\to Z_+ \mid p \in \sC\}$ with 
\begin{equation}
\label{eq: Mukai vector v+}
(\rk \cF_+, c_{1}(\cF_+),c_{2}(\cF_+))=v_{Z_+}
\end{equation}
parametrised by $\sC$ itself.

One crucial feature of the building block obtained from $X_+=\P^1\times\P^2$ is the fact that the moduli space of bundles over the anti-canonical $K3$ divisor $\kd_+$ is actually isomorphic to $\kd_+$ itself: 

\begin{prop}[{\cite[Lemma 4.7 \& Proposition 4.8]{Menet2017}}]\label{iso}
For each $p\in \kd_+$, there exists an  $\mathcal{A}_+$-$\mu$-stable and  rank $2$ Hartshorne-Serre  bundle $\mathcal{G}_p\to \kd_+$ obtained from $p$.
The induced map 
$$
\begin{array}{rcl}
g:\kd_+&\longrightarrow&\sM_{\kd_+,\mathcal{A}_+}^{s}(v_{\kd_+})\\
p&\longmapsto&\mathcal{G}_p
\end{array}
$$ is an isomorphism of $K3$ surfaces.

\end{prop}

Now let
$\mathcal{G} \in  \sM_{\kd_+,\mathcal{A}_+}^{s}(v_{\kd_+})$ and
 $V \subset H^1(\kd_+, \Endr_0(\mathcal{G}))$.
 From Proposition \ref{iso}, there is $p\in\kd_+$ such that $\mathcal{G}=\mathcal{G}_p$ and let $V'=(\rd g)_p ^{-1}(V)$.
 Since ${-}K_{X_+|\kd_+}$ is very ample (see Example \ref{ex:2conics}), Lemma \ref{lem:wiggle} allows the choice of
a smooth base locus curve $\sC \in |{-}K_{X_+|\kd_+}|$ such that $p\in \sC$ and $T_p \sC = V'$.
By Proposition \ref{prop:exa}, we can find a family  $\{(\cF_+)_q\to Z \mid q \in \sC\}$ of holomorphic bundles parametrised by $\sC$, with prescribed topology (\ref{eq: Mukai vector v+}) and $(\cF_{+|\kd})_q=\mathcal{G}_q$
.
Such a bundle $\cF_+$ has therefore all the properties claimed in Theorem  \ref{thm:summaryintro}.

\begin{cor}[{\cite[Corollary 6.1]{Menet2017}}]
\label{cor:summarybis}
In the context of Example \ref{ex:2conics}, for every bundle
$\mathcal{G} \in  \sM_{\kd_+,\mathcal{A}_+}^{s}(v_{\kd_+})$ and
every complex line $V \subset H^1(\kd_+, \Endr_0(\mathcal{G}))$, there are
a smooth curve $\sC_+ \in |{-}K_{X_+|\kd_+}|$ and an asymptotically stable
and inelastic vector bundle $\cE_+ \to Z_+$ such that $\cE_{+|\kd_+} = \mathcal{G}$
and the restriction map
$$
\res : H^1(Z_+, \Endr_0(\cE_+)) \to H^1(\kd_+, \Endr_0(\mathcal{G}))
$$ 
has
image $V$.
\end{cor}
Let 
$$
\cE_-:=\cF_-\otimes\mathcal{O}_{Z_-}(-H_-+2G_-).
$$
\begin{cor}[{\cite[Corollary 6.2]{Menet2017}}]
\label{cor:summarybis2}
In the context of Example \ref{ex:2conics2}, there exists a family of asymptotically
stable and inelastic vector bundles $\{\cE_- \to Z_-\}$, parametrised by the
set of the lines in $X_-$ of class $h_-$, such that 
$\cE_{-|\kd_-}\in\sM_{\kd_-,\mathcal{A}_-}^{s}(v_{\kd_-})$.
\end{cor}

We fix a representative $\cE_- \to Z_-$ in the family of holomorphic bundles
 from Corollary  \ref{cor:summarybis2},
to be matched by a bundle $\cE_+ \to Z_+$ given by Corollary \ref{cor:summarybis},
so that asymptotic stability and inelasticity hold from the outset.

It remains to address compatibility and transversality.
Since the chosen configuration for $\hkr$ ensures that
$\hkr^*$ identifies the Mukai vectors of $\cE_{\pm|\kd_\pm}$, 
it induces a map $\bar \hkr^* :\sM_{\kd_-,\mathcal{A}_-}^{s}(v'_{\kd_-})\to
\sM_{\kd_+,\mathcal{A}_+}^{s}(v'_{\kd_+})$.
In particular, the target moduli space is $2$-dimensional, by Proposition
\ref{qui}, and $\hkr^* (\im \res_-)$ is  $1$-dimensional, since the bundles
$\{\cE_-\}$ are parametrised by lines of fixed class $h_-$. So indeed we apply
Corollary \ref{cor:summarybis} with $\mathcal{G}= \bar \hkr^* (\cE_{-|\kd_-})$
and any choice of a direct complement subspace $V$ such that 
$$
V\oplus \bar
\hkr^* (\im \res_-)=H^1\left(\kd_+, \Endr_0\left(\bar \hkr^* (\cE_{-|\kd_-})\right)\right).
$$

Denoting by $\cM_{\kd_\pm}(v)$ the moduli space of ASD instantons over
$\kd_\pm$ with Mukai vector $v$,
the maps $f_\pm$ (cf. (\ref{eq: isomorphism
f})) in Theorem \ref{thm:HenriqueThomas}  are the linearisations of the Hitchin-Kobayashi isomorphisms 
$$
\sM^{s}_{\kd_\pm,\mathcal{A}_\pm}(v'_{\kd\pm})
\simeq
\cM_{\kd_\pm}(v'_{\kd\pm}).
$$
Therefore, our bundles  $\cE_\pm$ indeed satisfy $A_{\infty,+}=\bar \hkr^*
A_{\infty,-}$ for the corresponding
instanton connections. Moreover, by linearity,   $\lambda_+(H^1(Z_+, \Endr_0(\cE_+)))$
is transverse in
$T_{A_{\infty,+}} \cM_{\kd_+}(v_{\kd_+}')$
 to the image of the real $2$-dimensional subspace  $\lambda_-(H^1(Z_-,\Endr_0(\cE_-)))
\subset
T_{A_{\infty,-}} \cM_{\kd_-}(v_{\kd_-}')$ under the linearisation of $\bar
\hkr^*$.

\bibliography{Bibliografia-2018-06}
\end{document}

%% file: init.tex
\newcommand{\rd}{{\rm d}}

\newcommand{\rE}{{\rm E}}

\newcommand{\rG}{{\rm G}}

\newcommand{\rU}{{\rm U}}

\newcommand{\ua}{{\underline a}}

\newcommand{\bm}{{\bf m}}

\newcommand{\bP}{{\bf P}}

\newcommand{\cE}{\mathcal{E}}
\newcommand{\cF}{\mathcal{F}}

\newcommand{\cH}{\mathcal{H}}

\newcommand{\cM}{\mathcal{M}}

\newcommand{\cT}{\mathcal{T}}

\newcommand{\cX}{\mathcal{X}}

\newcommand{\sA}{\mathscr{A}}
\newcommand{\sB}{\mathscr{B}}
\newcommand{\sC}{\mathscr{C}}

\newcommand{\sG}{\mathscr{G}}

\newcommand{\sM}{\mathscr{M}}

\newcommand{\sS}{\mathscr{S}}

\newcommand{\sY}{\mathscr{Y}}

\newcommand{\fg}{{\mathfrak g}}

\newcommand{\fr}{{\mathfrak r}}

\newcommand{\N}{\mathbb{N}}
\newcommand{\Z}{\mathbb{Z}}

\newcommand{\R}{\mathbb{R}}

\newcommand{\U}{{\rm U}}
\newcommand{\PU}{{\bP\U}}

\newcommand{\Pic}{{\rm Pic}}
\renewcommand{\P}{\mathbb{P}}

\newcommand{\Crit}{\mathrm{Crit}}

\newcommand{\cEnd}{\cE nd}

\renewcommand{\epsilon}{\varepsilon}

\newcommand{\Hol}{\mathrm{Hol}}

\newcommand{\delbar}{\bar \del}
\newcommand{\del}{\partial}

\newcommand{\im}{\mathop{\mathrm{im}}}
\renewcommand{\Im}{\mathop{\mathrm{Im}}}

\newcommand{\into}{\hookrightarrow}
\newcommand{\iso}{\cong}

\renewcommand{\Re}{\mathop{\mathrm{Re}}}

\newcommand{\spec}{\mathrm{spec}}

\newcommand{\tr}{\mathop{\mathrm{tr}}\nolimits}
\newcommand{\rk}{\mathop{\mathrm{rk}}\nolimits}

\newcommand{\Ext}{\mathrm{Ext}}

\newcommand{\Bl}{\mathrm{Bl}}

\newcommand{\qandq}{\quad\text{and}\quad}
\newcommand{\qwithq}{\quad\text{with}\quad}

\newcommand{\co}{\mskip0.5mu\colon\thinspace}
\def\<{\mathopen{}\left<}
\def\>{\right>\mathclose{}}
\def\({\mathopen{}\left(}
\def\){\right)\mathclose{}}
    
\def\wtilde{\widetilde}

\usepackage{multicol, color}

\definecolor{gold}{rgb}{0.85,.66,0}
\definecolor{cherry}{rgb}{0.9,.1,.2}
\definecolor{burgundy}{rgb}{0.8,.2,.2}
\definecolor{orangered}{rgb}{0.85,.3,0}
\definecolor{orange}{rgb}{0.85,.4,0}
\definecolor{olive}{rgb}{.45,.4,0}
\definecolor{lime}{rgb}{.6,.9,0}
\definecolor{green}{rgb}{.2,.7,0}
\definecolor{grey}{rgb}{.4,.4,.2}
\definecolor{brown}{rgb}{.4,.3,.1}

\newtheorem{theorem}[equation]{Theorem}
\newtheorem{prop}[equation]{Proposition}
\newtheorem{cor}[equation]{Corollary}

\newtheorem{lemma}[equation]{Lemma}

\numberwithin{substep}{step}

\numberwithin{subcase}{case}

\theoremstyle{remark}
\newtheorem{remark}[equation]{Remark}

\theoremstyle{definition}
\newtheorem{definition}[equation]{Definition}
\newtheorem{example}[equation]{Example}

\newtheorem{summary}[equation]{Summary}

%% file: abs.tex
\begin{abstract}
  We review a method to construct $\rG_2$--instantons over compact $\rG_2$--manifolds arising as the twisted connected sum of a matching pair of Calabi-Yau $3$-folds with cylindrical end, based on the series of articles \cite{SaEarp2015a,SaEarp2015b,Jardim2017,Menet2017} by the author and others.
  The construction is based on gluing  $\rG_2$--instantons obtained from holomorphic bundles over such building blocks, subject to natural compatibility and transversality conditions. Explicit examples are obtained from matching pairs of semi-Fano $3$-folds by an algorithmic procedure based on the Hartshorne-Serre correspondence.
\end{abstract}
